\documentclass[12 pt]{amsart}
\usepackage{amscd,amsfonts,amssymb,amsmath}
\newtheorem{theorem}{Theorem}[section]

\newtheorem{lemma}[theorem]{Lemma}
\newtheorem{proposition}[theorem]{Proposition}
\theoremstyle{definition}

\newtheorem{question}[theorem]{Question}

\numberwithin{equation}{subsection}
\newtheorem*{ack}{Acknowledgements}

\newcommand{\Z}{\mathbb{Z}}

\usepackage{graphicx}   
\newcommand{\secref}[1]{Section~\ref{#1}}
\newcommand{\thmref}[1]{Theorem~\ref{#1}}

\begin{document}
\title[On some decompositions of the  3-strand Singular  Braid Group]{On some decompositions of the  3-strand Singular  Braid Group}
\author[K. Gongopadhyay]{Krishnendu Gongopadhyay}
\author[T. A. Kozlovskaya]{Tatyana A.  Kozlovskaya} 
\author[O. V. Mamonov]{Oleg V.  Mamonov}
\address{Indian Institute of Science Education and Research (IISER) Mohali,
 Knowledge City,  Sector 81, S.A.S. Nagar 140306, Punjab, India}
\email{krishnendu@iisermohali.ac.in}
\address{Regional Scientific and Educational Mathematical Center of Tomsk State University, 36 Lenin Ave., Tomsk, Russia.}
\email{t.kozlovskaya@math.tsu.ru}
\address{Novosibirsk State Agrarian University,
Dobrolyubova street, 160, Novosibirsk, 630039, Russia}
\email{mmnv20@yandex.ru}

\maketitle 

\begin{abstract}
Let $SB_n$ be the singular braid group generated by  braid generators $\sigma_i$ and singular braid generators $\tau_i$, $1 \leq i \leq n-1$. Let $ST_n$ denote the group that is the kernel of the homomorphism that maps, for each $i$,  $\sigma_i$ to the cyclic permutation $(i, i+1)$ and $\tau_i$ to $1$. In this paper we investigate the group $ST_3$. We obtain a presentation for $ST_3$. We prove that  $ST_3$  is isomorphic to the singular pure braid group $SP_3$ on $3$ strands. We also prove that the group $ST_3$ is semi-direct product of a subgroup $H$ and an infinite cyclic group, where  the subgroup $H$ is an  HNN-extension of $\Z^2 \ast \Z^2$. 
\end{abstract}


\section{Introduction}
The notion of singular braids was introduced independently by Baez in \cite{bae} and Birman in \cite{bir}. The set of all such braids has a monoid structure.  It was shown in \cite{fkr} that the Baez-Birman monoid on $n$ strands is embedded in a group which is denoted by $SB_n$. The group $SB_n$ is now known as the \emph{singular braid group} on $n$ strands. The group $SB_n$ contains the classical braid group $B_n$ as a subgroup.  The \emph{singular braid group} $SB_n$ is generated by a set of $2(n-1)$ generators: $\{  \sigma_i,~ \tau_i ~ | ~ i=1, 2, \ldots, n-1  \}$, where $\sigma_i$  satisfy the usual braid relations: 
$$\sigma_i \sigma_j=\sigma_j \sigma_i, \hbox{ if } |i-j|>1;$$
	$$\sigma_i \sigma_{i+1} \sigma_i=\sigma_{i+1} \sigma_i \sigma_{i+1},$$ 
and  $\tau_i$ satisfy the commuting relations: 
$$\tau_i \tau_j = \tau_j \tau_i, \hbox{ if } |i-j|>1;$$
and in addition there are the following mixed relations among $\sigma_i$, $\tau_i$: 
		 \begin{equation}  \label{mm2}  \sigma_{i+1} \sigma_i \tau_{i+1} =\tau_i \sigma_{i+1} \sigma_i;\end{equation}
	 \begin{equation}  \label{mm3}  \sigma_i \sigma_{i+1} \tau_i =\tau_{i+1} \sigma_i \sigma_{i+1}.\end{equation}
 \begin{equation} \label{mm1} \tau_i \sigma_j = \sigma_j \tau_i,  \hbox{ if } i=j \hbox{ or } |i-j|>1; \end{equation}
The generators include the standard  braids $\sigma_i$ and  braids $\tau_i$ (see Fig. \ref{1ab}, \ref{1c}).

\begin{figure}[h]
\includegraphics[totalheight=2.5cm]{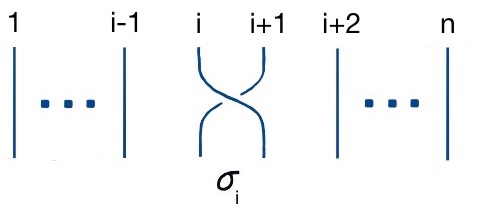} \,\,\,\,\,\,\,\,
\includegraphics[totalheight=2.5cm]{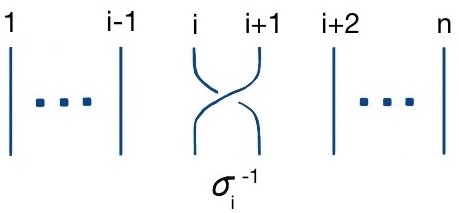}
\caption{The elementary braids $\sigma_i$ and  $\sigma_i^{-1}$} \label{1ab}
\end{figure}

\begin{figure}[h]
\centering{
\includegraphics[totalheight=2.5cm]{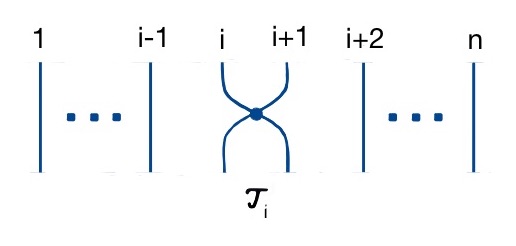}\\
\caption{The singular  braids $\tau_i$} \label{1c}
}
\end{figure}

Singular braids are related to finite type invariants of knots and links. It is a natural problem to investigate their algebraic and geometric properties to understand these invariants. The word problem for $SB_3$ was solved in \cite{j}, \cite{dg1}. For arbitrary $n$, it follows from the work of Corran \cite{co} or Godelle and Paris \cite{pa}. For more information on generalised braids and singular braid groups, we refer to the survey \cite{v}. 

\medskip In \cite{dg2}, Dasbach and  Gemein introduced the singular pure braid group $SP_n$ that is a generalization of the (classical) pure braid group $P_n$. The group $SP_n$ is the kernel of the natural surjective homomorphism that maps, for each $i$,  $\sigma_i$ and $\tau_i$ to the cyclic permutation $(i, i+1)$.  Dasbach and Gemein found a set of generators and defining relations for $SP_n$ and established that this group can be constructed using consecutive HNN extensions. Recently, Bardakov and Kozlovskaya \cite{bk}  revisited $SP_3$ and obtained another presentation for it that decomposes $SP_3$ as a direct product of two groups.

For the virtual braid group $VB_n$ people study  the kernels of two homomorphisms: $\varphi_1, \varphi_2 : VB_n \to S_n$. The first is defined by the rules
$$
\varphi_1(\sigma_i) = \varphi_1(\rho_i) = (i, i+1), ~~i = 1, 2, \ldots, n-1,
$$
and the kernel  $Ker(\varphi_1)$ is called the \emph{virtual pure braid group} and is denoted $VP_n$. This group was introduced in \cite{B}.  The second homomorphism is defined by the rules
$$
\varphi_2(\sigma_i) = e,~~\varphi_2(\rho_i) = (i, i+1), ~~i = 1, 2, \ldots, n-1,
$$
and the kernel  $Ker(\varphi_2)$ is called the \emph{Rabenda group} and is denoted $VR_n$. This group was introduced in \cite{R}. 
 In \cite{BB}, it was proved that the group $VP_n$ is not isomorphic to $VR_n$ for $n \geq 3$.

 Consider the homomorphism
$$
\pi : SB_n \longrightarrow S_n
$$
of $SB_n$ onto the symmetric group $S_n$ on $n$ symbols by actions on the  generators
\begin{center}
$\pi(\sigma_i) =s_i=(i, i+1)$, $i=1, 2, \ldots, n-1$, 
$\pi(\tau_j) = 1$, $j=1, 2, \ldots, n-1$. 
\end{center}
Hence, we have decomposition
$$
1 \to \mbox{Ker}(\pi) \to SB_n \to S_n \to 1.
$$
Denote by $ST_n$ the kernel $\mbox{Ker}(\pi)$. So, the group $ST_n$ may be thought of as an opposite analogue of the Rabenda group.

In this paper, we obtain a presentation for the group $ST_3$. Using this presentation, we prove that the group $ST_3$ is a semi-direct product of a subgroup $H$ and an infinite cyclic group, where  the subgroup $H$ is an  HNN-extension of $\Z^2 \ast \Z^2$.  Further, by comparing the presentation of $ST_3$ and that of $SP_3$ obtained in \cite{bk} we have the following. 
\begin{theorem}\label{th1}  
The group $ST_3$ is isomorphic to the  singular pure braid group $SP_3$. 
\end{theorem}

\medskip 
We prove this theorem in \secref{pd}. The semidirect decomposition has also been proved in this section.  This result rely on a presentation for $ST_3$, see \thmref{t1}, that is obtained by using the Reidemeister-Schreier method in \secref{pf}. 

\medskip In the general case we can formulate

\begin{question}
Is it true that $SP_n$ is isomorphic to $ST_n$ for $n > 3$.
\end{question}

\section{Reidemeister-Schreier Algorithm} Given a presentation of a group $G$, this algorithm allows one to find a presentation of a subgroup $H\subset G$. To obtain the presentation of $H$,   it is necessary to find a  Schreier's set of right coset of the group $G$ over the subgroup $H$. We briefly recall the algorithm. 
Let $a_1,\ldots,a_n$ be the generators of the group $G$ and $R_1,\ldots,R_m$ be
the set of defining relations for the given set of generators.
A set of words $N=\left\{ K_\alpha,\, \alpha\in A  \right\}$
on generators $a_1,\ldots,a_n$  defines a Schreier's system for the subgroup $H\subset G$
relative to the system of generators $a_1,\ldots,a_n$
if the following conditions are satisfied:

\medskip 1) There is only one word of $N$ from every right coset of the group $G$ over $H$. 

\medskip 2) If the word
$K_\alpha=a_{i_1}^{\varepsilon_1}\ldots a_{i_{p-1}}^{\varepsilon_{p-1}}a_{i_p}^{\varepsilon_p}$,
$(\varepsilon_j=\pm 1)$ lies in $N$,
then the word $a_{i_1}^{\varepsilon_1}\ldots a_{i_{p-1}}^{\varepsilon_{p-1}}$
also lies in $N$.

\medskip Suppose that some
Schreier's system $N$ is chosen for the subgroup $H\subset G$
relative to the system generators $a_1,\ldots,a_n$ of $G$.
For every word $Q$ on $a_1,\ldots,a_n$, we denote by $\overline{Q}$
the only word from $N$ which lies in the same right coset of $G$ over the subgroup $H$.
Denote
$$
S_{K_\alpha, a_\nu}=K_\alpha a_\nu \cdot (\overline{K_\alpha a_\nu})^{-1},
\quad \alpha\in A,\,\,\nu=1,\ldots,n.
$$
Theorem of Reidemeister-Schreier states that the elements $S_{K_\alpha, a_\nu}$
generate subgroup $H$ and
the set of defining relations for this set of generators
is given by the following. First set of relation consists of trivial relations $S_{K_\alpha, a_\nu}=1$,
where the pair $K_\alpha$, $a_\nu$ is such that the word
$K_\alpha a_\nu \cdot (\overline{K_\alpha a_\nu})^{-1}$
is freely equivalent to the word 1.
Second set of relations consists of all relations of the form $\tau(K_\alpha R_\mu K_\alpha^{-1})$,
where $\alpha\in A$, $\mu=1,\ldots,m$, and $\tau$ is Reidemeister's transformation,
which  maps every nonempty word
$a_{i_1}^{\varepsilon_1}\ldots a_{i_p}^{\varepsilon_p}$, $(\varepsilon_j=\pm 1)$
from symbols $a_1,\ldots,a_n$ to the word from symbols $S_{K_\alpha, a_\nu}$
by the rule:
$$
\tau (a_{i_1}^{\varepsilon_1}\ldots a_{i_p}^{\varepsilon_p})=
 S_{K_{i_1}, a_{i_1}}^{\varepsilon_1} \ldots S_{K_{i_p}, a_{i_p}}^{\varepsilon_p},
$$
where $K_{i_j}=\overline{a_{i_1}^{\varepsilon_1}\ldots a_{i_{j-1}}^{\varepsilon_{j-1}}}$,
if $\varepsilon_j=1$, and
$K_{i_j}=\overline{a_{i_1}^{\varepsilon_1}\ldots a_{i_{j}}^{\varepsilon_{j}}}$,
if $\varepsilon_j=-1$.

\section{Presentation of  $ST_3$} \label{pf} 
The group  $SB_3$ is generated by elements
$$
\sigma_1, \sigma_2, \tau_1, \tau_2,
$$
and is defined by relations
$$
\sigma_1 \tau_1 = \tau_1 \sigma_1,~~~\sigma_1 \sigma_2 \sigma_1 = \sigma_2 \sigma_1 \sigma_2,~~~\sigma_2 \tau_2 = \tau_2 \sigma_2,~~~
\sigma_1 \sigma_2 \tau_1 = \tau_2 \sigma_1 \sigma_2,~~~\sigma_2 \sigma_1 \tau_2 = \tau_1 \sigma_2 \sigma_1.
$$
The set of coset representatives:
$$
\Lambda_3 = \{ 1, \sigma_1, \sigma_2, \sigma_1 \sigma_2, \sigma_2 \sigma_1, \sigma_1 \sigma_2 \sigma_1  \}.
$$
The group $ST_3$ is generated by elements
$$
S_{\lambda,a} = \lambda a \cdot (\overline{\lambda a})^{-1},~~~\lambda \in \Lambda_3,~~a \in \{ \sigma_1, \sigma_2, \tau_1, \tau_2 \}.
$$
Find these elements
\begin{align*}
& S_{1,\sigma_1} = \sigma_1 \cdot (\overline{\sigma_1})^{-1} = \sigma_1 \cdot \sigma_1^{-1} = 1,\\
& S_{1,\sigma_2} = \sigma_2 \cdot (\overline{\sigma_2})^{-1} = \sigma_2 \cdot \sigma_2^{-1} = 1,\\
& S_{1,\tau_1} = \tau_1 \cdot (\overline{\tau_1})^{-1} = \tau_1 ,\\
& S_{1,\tau_2} = \tau_2 \cdot (\overline{\tau_2})^{-1} = \tau_2 ,\\
\end{align*}
\begin{align*}
& S_{\sigma_1,\sigma_1} = \sigma_1^2 \cdot \overline{\sigma_1^2}^{-1} = \sigma_1^2 \cdot 1 = \sigma_1^2,\\
& S_{\sigma_1,\sigma_2} = \sigma_1 \sigma_2 \cdot (\overline{\sigma_1 \sigma_2})^{-1} = 1,\\
& S_{\sigma_1,\tau_1} = \sigma_1 \tau_1 \cdot (\overline{\sigma_1 \tau_1})^{-1} =  \tau_1,\\
& S_{\sigma_1,\tau_2} = \sigma_1 \tau_2 \cdot (\overline{\sigma_1 \tau_2})^{-1} = \sigma_1 \tau_2  \sigma_1^{-1},\\
\end{align*}
\begin{align*}
& S_{\sigma_2,\sigma_1} = \sigma_2 \sigma_1 \cdot (\overline{\sigma_2 \sigma_1})^{-1} = 1,\\
& S_{\sigma_2,\sigma_2} = \sigma_2^2 \cdot \overline{\sigma_2^2}^{-1} = \sigma_2^2 \cdot 1 = \sigma_2^2,\\
& S_{\sigma_2,\tau_1} = \sigma_2 \tau_1 \cdot (\overline{\sigma_2 \sigma_1})^{-1} = \sigma_2 \tau_1  \sigma_2^{-1},\\
& S_{\sigma_2,\tau_2} = \sigma_2 \tau_2 \cdot (\overline{\sigma_2 \tau_2})^{-1} =  \tau_2,\\
\end{align*}
\begin{align*}
& S_{\sigma_1 \sigma_2,\sigma_1} = \sigma_1 \sigma_2 \sigma_1 \cdot (\sigma_1 \sigma_2 \sigma_1)^{-1} = 1,\\
& S_{\sigma_1 \sigma_2,\sigma_2} = \sigma_1 \sigma_2^2 \sigma_1^{-1},\\
& S_{\sigma_1 \sigma_2,\tau_1} = \sigma_1 \sigma_2 \tau_1 \sigma_2^{-1} \sigma_1^{-1}=\tau_2,\\
&S_{\sigma_1 \sigma_2,\tau_2} = \sigma_1 \tau_2 \sigma_1^{-1},\\
\end{align*}
\begin{align*}
& S_{\sigma_2 \sigma_1,\sigma_1} = \sigma_2 \sigma_1^2 \sigma_2^{-1},\\
& S_{\sigma_2 \sigma_1,\sigma_2} = \sigma_2 \sigma_1 \sigma_2 \sigma_1^{-1} \sigma_2^{-1} \sigma_1^{-1}=1,\\
& S_{\sigma_2 \sigma_1,\tau_1} = \sigma_2  \tau_1 \sigma_2^{-1},\\
& S_{\sigma_2 \sigma_1,\tau_2} = \sigma_2 \sigma_1 \tau_2 \sigma_1^{-1} \sigma_2^{-1}=\tau_1,\\
\end{align*}
\begin{align*}
& S_{\sigma_1 \sigma_2 \sigma_1, \sigma_1} = \sigma_1 \sigma_2 \sigma_1^2 \sigma_2^{-1} \sigma_1^{-1}=\sigma_2^{2} ,\\
& S_{\sigma_1 \sigma_2 \sigma_1, \sigma_2} = \sigma_1 \sigma_2 \sigma_1 \sigma_2 \sigma_1^{-1} \sigma_2^{-1}=\sigma_1^{2},\\
& S_{\sigma_1 \sigma_2 \sigma_1, \tau_1} = \sigma_1 \sigma_2 \tau_1 \sigma_2^{-1} \sigma_1^{-1}=\tau_2,\\
& S_{\sigma_1 \sigma_2 \sigma_1, \tau_2} = \sigma_1 \sigma_2 \sigma_1 \tau_2 \sigma_1^{-1} \sigma_2^{-1} \sigma_1^{-1}=\tau_1.\\
\end{align*}

Find the set of defining relations.

\begin{lemma} \label{l1}
From relation $r_1 = \sigma_1 \tau_1 \sigma_1^{-1} \tau_1^{-1}$ follows 6 relations, applying which we can remove generators:
$$S_{\sigma_1,\tau_1} = S_{1,\tau_1}; 
S_{\sigma_2\sigma_1, \tau_1} = S_{\sigma_2, \tau_1};
S_{\sigma_1 \sigma_2 \sigma_1,\tau_1} =S_{\sigma_1 \sigma_2, \tau_1},
$$

and we get 3 relations:

$$S_{\sigma_1,\sigma_1} S_{\sigma_1,\tau_1} =S_{\sigma_1, \tau_1}S_{\sigma_1, \sigma_1}, $$
$$S_{\sigma_2 \sigma_1, \sigma_1}S_{\sigma _2,\tau _1}= S_{\sigma_2, \tau_1}S_{\sigma _2\sigma _1, \sigma _1}, $$
$$S_{\sigma_1 \sigma_2 \sigma_1, \sigma_1}S_{\sigma_1 \sigma_2,\tau_1} = S_{\sigma_1 \sigma_2, \tau_1} S_{\sigma_1 \sigma_2 \sigma_1, \sigma_1}.$$

\end{lemma}
\begin{proof} Take the relation $r_1 = \sigma_1 \tau_1 \sigma_1^{-1} \tau_1^{-1}$. 
From this, we get the following relations.  
$$\tau(r_1) = S_{1,\sigma_1} S_{\sigma_1,\tau_1}  S_{1,\sigma_1}^{-1} S_{1,\tau_1}^{-1}=S_{\sigma_1,\tau_1}  S_{1,\tau_1}^{-1}=1,$$
this implies, 
$$S_{\sigma_1,\tau_1} = S_{1,\tau_1}$$

$$
r_{1,\sigma_1} = \tau(\sigma_1 r_1 \sigma_1^{-1})= S_{1,\sigma_1} S_{\sigma_1,\sigma_1} S_{1, \tau_1} S_{\sigma_1, \sigma_1}^{-1} S_{\sigma_1, \tau_1}^{-1} S_{1,\sigma_1}^{-1} =  S_{\sigma_1,\sigma_1} S_{1, \tau_1} S_{\sigma_1 \sigma_1}^{-1} S_{\sigma_1, \tau_1}^{-1} = 1,
$$
this implies, 
$$
S_{\sigma_1,\sigma_1} S_{1, \tau_1} =S_{\sigma_1, \tau_1}S_{\sigma_1, \sigma_1} 
$$ 

$$
r_{1,\sigma_2} =\tau( \sigma_2 r_1 \sigma_2^{-1}) = S_{1,\sigma_2} S_{\sigma_2,\sigma_1} S_{\sigma_2 \sigma_1, \tau_1} S_{\sigma_2, \sigma_1}^{-1} S_{\sigma_2, \tau_1}^{-1} S_{1,\sigma_2}^{-1} =  S_{\sigma_2\sigma_1, \tau_1}  S_{\sigma_2, \tau_1}^{-1} = 1,
$$
this implies, 
$$
S_{\sigma_2\sigma_1, \tau_1} = S_{\sigma_2, \tau_1}. 
$$

$$
r_{1,\sigma_1\sigma_2} =  \tau(\sigma_1 \sigma_2 r_1 \sigma_2^{-1} \sigma_1^{-1}) = S_{1,\sigma_1} S_{\sigma_1,\sigma_2} S_{\sigma_1 \sigma_2, \sigma_1} S_{\sigma_1 \sigma_2 \sigma_1,\tau_1} S_{\sigma_1 \sigma_2, \sigma_1}^{-1} S_{\sigma_1 \sigma_2, \tau_1}^{-1} S_{\sigma_1, \sigma_2}^{-1} S_{1,\sigma_1}^{-1} =
$$
$$
= S_{\sigma_1 \sigma_2 \sigma_1,\tau_1}  S_{\sigma_1 \sigma_2, \tau_1}^{-1}  = 1,
$$
this implies, 
$$
S_{\sigma_1 \sigma_2 \sigma_1,\tau_1} =S_{\sigma_1 \sigma_2, \tau_1}.
$$
From the relation
$${r_{1,{\sigma _2}{\sigma _1}}} =\tau( \sigma_2 \sigma_1 r_1 \sigma_1^{-1} \sigma_2^{-1}) = S_{1,\sigma_2} S_{\sigma_2,\sigma_1} S_{\sigma_2 \sigma_1, \sigma_1} S_{\sigma_2 ,\tau_1} S_{\sigma_2 \sigma_1, \sigma_1}^{-1} S_{\sigma_2 \sigma_1, \tau_1}^{-1} S_{\sigma_2, \sigma_1}^{-1} S_{1,\sigma_2}^{-1} =$$

$$=S_{\sigma_2 \sigma_1, \sigma_1}S_{\sigma_2 ,\tau_1} S_{\sigma_2 \sigma_1, \sigma_1}^{-1} S_{\sigma_2 \sigma_1, \tau_1}^{-1} =1,$$
it follows that

$$S_{\sigma_2 \sigma_1, \sigma_1}S_{\sigma _2,\tau _1}= S_{\sigma _2\sigma _1,\tau _1}S_{\sigma _2\sigma _1, \sigma _1}.$$

Relation
$$
r_{1,\sigma_1\sigma_2\sigma_1} =\tau( \sigma_1 \sigma_2 \sigma_1 r_1 \sigma_1^{-1} \sigma_2^{-1} \sigma_1^{-1}) =
S_{1,\sigma_1} S_{\sigma_1,\sigma_2} S_{\sigma_1 \sigma_2, \sigma_1} S_{\sigma_1 \sigma_2 \sigma_1, \sigma_1}
  S_{\sigma_1 \sigma_2,\tau_1} \cdot
  $$
$$
\cdot  S_{\sigma_1 \sigma_2 \sigma_1, \sigma_1}^{-1} S_{\sigma_1 \sigma_2 \sigma_1, \tau_1}^{-1} S_{\sigma_1 \sigma_2, \sigma_1}^{-1} S_{\sigma_1, \sigma_2}^{-1}  S_{1,\sigma_1}^{-1} =
  S_{\sigma_1 \sigma_2 \sigma_1, \sigma_1}
  S_{\sigma_1 \sigma_2,\tau_1} S_{\sigma_1 \sigma_2 \sigma_1, \sigma_1}^{-1} S_{\sigma_1 \sigma_2 \sigma_1, \tau_1}^{-1} = 1,
$$

\medskip 
we get the relation
$$
S_{\sigma_1 \sigma_2 \sigma_1, \sigma_1}
  S_{\sigma_1 \sigma_2,\tau_1} = S_{\sigma_1 \sigma_2 \sigma_1, \tau_1} S_{\sigma_1 \sigma_2 \sigma_1, \sigma_1}.
$$
Now the lemma follows. 
\end{proof}

\begin{lemma} \label{l2}
From relation $r_2 = \sigma_1 \sigma_2 \sigma_1 \sigma_2^{-1} \sigma_1^{-1} \sigma_2^{-1}$ follows 6 relations, applying which we can remove 3 generators:
$$S_{\sigma_2 \sigma_1, \sigma_2} = 1,~~~ S_{\sigma_1, \sigma_1}=S_{\sigma_1 \sigma_2 \sigma_1, \sigma_2}, ~~~S_{\sigma_1 \sigma_2 \sigma_1, \sigma_1} = S_{\sigma_2, \sigma_2},
$$
and we get relations:
$${S_{{\sigma _1},{\sigma _1}}}{S_{\sigma _2\sigma _1,\sigma _1}} = {S_{\sigma _1\sigma _2,\sigma _2}}{S_{\sigma _1,\sigma _1}},$$

$${S_{\sigma _2\sigma _1,\sigma _1}}{S_{\sigma _2,\sigma _2}} = {S_{\sigma _2,\sigma _2}} {S_{\sigma _1 \sigma _2,\sigma _2}},$$

$${S_{{\sigma _2},{\sigma _2}}}{S_{{\sigma _1}{\sigma _2},{\sigma _2}}}{S_{\sigma _1 \sigma _1}}= {S_{{\sigma _1},{\sigma _1}}}{S_{{\sigma _2}{\sigma _1},{\sigma _1}}}{S_{{\sigma _2},{\sigma _2}}}.$$
\end{lemma}

\begin{proof}
 Take the relation $r_2 = \sigma_1 \sigma_2 \sigma_1 \sigma_2^{-1} \sigma_1^{-1} \sigma_2^{-1}$. Then
$$
r_2 = r_{2,1} = S_{1,\sigma_1} S_{\sigma_1,\sigma_2} S_{\sigma_1 \sigma_2, \sigma_1} S_{\sigma_2 \sigma_1, \sigma_2}^{-1}
S_{\sigma_2, \sigma_1}^{-1} S_{1, \sigma_2}^{-1} =  S_{\sigma_2 \sigma_1, \sigma_2}^{-1} = 1,
$$
i.e. $S_{\sigma_2 \sigma_1, \sigma_2} = 1$ and we can remove this generator.

Conjugating this relation by $\sigma_1^{-1}$, we get
$$
r_{2,\sigma_1} =  S_{1,\sigma_1} S_{\sigma_1,\sigma_1} S_{1,\sigma_2} S_{\sigma_2, \sigma_1} S_{\sigma_1 \sigma_2 \sigma_1, \sigma_2}^{-1}
S_{\sigma_1 \sigma_2, \sigma_1}^{-1} S_{\sigma_1, \sigma_2}^{-1} S_{1, \sigma_1}^{-1} =  S_{\sigma_1, \sigma_1} S_{\sigma_1 \sigma_2 \sigma_1, \sigma_2}^{-1}= 1,
$$
i.e. $ S_{\sigma_1, \sigma_1}=S_{\sigma_1 \sigma_2 \sigma_1, \sigma_2}$.

Conjugating $r_2$ by $\sigma_2^{-1}$, we get
$$
r_{2,\sigma_2} =  S_{1,\sigma_2} S_{\sigma_2,\sigma_1}  S_{\sigma_2 \sigma_1, \sigma_2} S_{\sigma_1 \sigma_2 \sigma_1, \sigma_1}
S_{\sigma_1, \sigma_2}^{-1} S_{1,\sigma_1}^{-1} S_{\sigma_2, \sigma_2}^{-1}  S_{1,\sigma_2}^{-1} =  S_{\sigma_2 \sigma_1, \sigma_2} S_{\sigma_1 \sigma_2 \sigma_1, \sigma_1}   S_{ \sigma_2, \sigma_2}^{-1} = 1.
$$
Since $S_{\sigma_2 \sigma_1, \sigma_2} = 1$, from this relation follows that
$S_{\sigma_1 \sigma_2 \sigma_1, \sigma_1} = S_{\sigma_2, \sigma_2}$ and we can remove $S_{\sigma_1 \sigma_2 \sigma_1, \sigma_1}$.

Conjugating $r_2$ by $(\sigma_1 \sigma_2)^{-1}$, we get

Relation
$${r_{2,{\sigma _1}{\sigma _2}}} = {S_{{\sigma _1}{\sigma _2}{\sigma _1},{\sigma _2}}}{S_{{\sigma _2}{\sigma _1},{\sigma _1}}}S_{{\sigma _1},{\sigma _1}}^{ - 1}S_{{\sigma _1}{\sigma _2},{\sigma _2}}^{ - 1} = 1$$
gives relation
$${S_{{\sigma _1}{\sigma _2}{\sigma _1},{\sigma _2}}}{S_{\sigma _2\sigma _1,\sigma _1}} = {S_{\sigma _1\sigma _2,\sigma _2}}{S_{\sigma _1,\sigma _1}}.$$

Conjugating by $(\sigma _2 \sigma _1)^{-1}$ we get

$${r_{2,{\sigma _2}{\sigma _1}}} = {S_{{\sigma _2}{\sigma _1},{\sigma _1}}}{S_{{\sigma _2},{\sigma _2}}}S_{{\sigma _1}{\sigma _2},{\sigma _2}}^{ - 1}S_{{\sigma _1}{\sigma _2}{\sigma _1},{\sigma _1}}^{ - 1} = 1$$
or
$${S_{{\sigma _2}{\sigma _1},{\sigma _1}}}{S_{{\sigma _2},{\sigma _2}}} = {S_{{\sigma _1}{\sigma _2}{\sigma _1},{\sigma _1}}}{S_{{\sigma _1}{\sigma _2},{\sigma _2}}},$$

$${S_{\sigma _2\sigma _1,\sigma _1}}{S_{\sigma _2,\sigma _2}} = {S_{\sigma _2,\sigma _2}} {S_{\sigma _1 \sigma _2,\sigma _2}}.$$

Take the relation:

\[{r_{2,{\sigma _1}{\sigma _2}{\sigma _1}}} = {S_{{\sigma _1}{\sigma _2}{\sigma _1},{\sigma _1}}}{S_{{\sigma _1}{\sigma _2},{\sigma _2}}}{S_{\sigma _1 \sigma _1}}S_{{\sigma _2},{\sigma _2}}^{ - 1}S_{{\sigma _2}{\sigma _1},{\sigma _1}}^{ - 1} S_{{\sigma _1}{\sigma _2}{\sigma _1},{\sigma _2}}^{ - 1}= 1,\]
that is equivalent to

\[{S_{{\sigma _2},{\sigma _2}}}{S_{{\sigma _1}{\sigma _2},{\sigma _2}}}{S_{\sigma _1 \sigma _1}}= {S_{{\sigma _1}{\sigma _2}{\sigma _1},{\sigma _2}}}{S_{{\sigma _2}{\sigma _1},{\sigma _1}}}{S_{{\sigma _2},{\sigma _2}}}\]
or

\[{S_{{\sigma _2},{\sigma _2}}}{S_{{\sigma _1}{\sigma _2},{\sigma _2}}}{S_{\sigma _1 \sigma _1}}= {S_{{\sigma _1},{\sigma _1}}}{S_{{\sigma _2}{\sigma _1},{\sigma _1}}}{S_{{\sigma _2},{\sigma _2}}}.\]
Now the lemma follows. 
\end{proof}

\begin{lemma} \label{l3}
From relation $r_3 = \sigma_2 \tau_2 \sigma_2^{-1} \tau_2^{-1}$ follows 6 relations, applying which we can remove 3 generators:
$$
{S_{1,{\tau _2}}} = {S_{{\sigma _2},{\tau _2}}},~~~{S_{{\sigma _1},{\tau _2}}} = {S_{{\sigma _1}{\sigma _2},{\tau _2}}},~~~ {S_{{\sigma _2}{\sigma _1},{\tau _2}}} = {S_{{\sigma _1}{\sigma _2}{\sigma _1},{\tau _2}}},
$$
and we get 3 relations:

\[ {S_{{\sigma _2},{\sigma _2}}} {S_{{\sigma _2},{\tau _2}}} = {S_{{\sigma _2},{\tau _2}}} {S_{\sigma_2, {\sigma _2}}},\]

\[{S_{\sigma _1 \sigma _2, \sigma _2}{S_{{\sigma _1}, {\tau _2}}}= {S_{{\sigma _1}, {\tau _2}}}S_{{\sigma _1}{\sigma _2},{\sigma _2}}},\]

\[{S_{\sigma _1,\sigma _1}} {S_{{\sigma _2}{\sigma _1},{\tau _2}}} =  {S_{{\sigma _2}{\sigma _1},{\tau _2}}} {S_{\sigma _1,\sigma _1}} .\]
\end{lemma}

\begin{proof} Consider the relation $r_3 = \sigma_2 \tau_2 \sigma_2^{-1} \tau_2^{-1}$. From it

\[{r_{3,1}} = {S_{{\sigma _2},{\tau _2}}}S_{1,{\tau _2}}^{ - 1}=1,\]

\[{S_{1,{\tau _2}}} = {S_{{\sigma _2},{\tau _2}}}.\]

Conjugating by $(\sigma_1)^{-1}$
\[{r_{3,{\sigma _1}}} = {S_{{\sigma _1}{\sigma _2},{\tau _2}}}S_{{\sigma _1},{\tau _2}}^{ - 1}\]
or
\[{S_{{\sigma _1},{\tau _2}}} = {S_{{\sigma _1}{\sigma _2},{\tau _2}}}.\]

Conjugating by $(\sigma_2)^{-1}$
\[{r_{3,{\sigma _2}}} = {S_{{\sigma _2},{\sigma _2}}}{S_{1,{\tau _2}}}S_{{\sigma _2},{\sigma _2}}^{ - 1} S_{{\sigma _2},{\tau _2}}^{ - 1} =1,\]
and we have the relation
\[ {S_{{\sigma _2},{\sigma _2}}} {S_{{1},{\tau _2}}} = {S_{{\sigma _2},{\tau _2}}} {S_{\sigma_2, {\sigma _2}}}.\]

Conjugating by $(\sigma _1 \sigma_2)^{-1}$
\[{r_{3,{\sigma _1}{\sigma _2}}} = {S_{{\sigma _1}{\sigma _2},{\sigma _2}}}{S_{{\sigma _1},{\tau _2}}}S_{{\sigma _1}{\sigma _2},{\sigma _2}}^{ - 1} S_{{\sigma _1}{\sigma _2},{\tau _2}}^{ - 1} = 1\]
or
\[{S_{\sigma _1 \sigma _2, \sigma _2}{S_{{\sigma _1}, {\tau _2}}}= {S_{{\sigma _1}{\sigma _2},{\tau _2}}}S_{{\sigma _1}{\sigma _2},{\sigma _2}}}.\]

Conjugating by $(\sigma _2 \sigma_1)^{-1}$
\[{r_{3,{\sigma _2}{\sigma _1}}} = {S_{{\sigma _1}{\sigma _2}{\sigma _1},{\tau _2}}}S_{{\sigma _2}{\sigma _1},{\tau _2}}^{ - 1} = 1.\]
We can remove the generator

\[{S_{{\sigma _2}{\sigma _1},{\tau _2}}} = {S_{{\sigma _1}{\sigma _2}{\sigma _1},{\tau _2}}}.\]

Conjugating by $(\sigma_1 \sigma _2 \sigma_1)^{-1}$
\[{r_{3,{\sigma _1}{\sigma _2}{\sigma _1}}} = {S_{\sigma _1  \sigma _2 \sigma _1 ,\sigma _2}} {S_{\sigma _2 \sigma _1 ,\tau _2}}  S_{\sigma _1  \sigma _2 \sigma _1 ,\sigma _2}^{ - 1} S_{{\sigma _1}{\sigma _2}{\sigma _1},{\tau _2}}^{ - 1} = 1\]
or
\[{S_{\sigma _1,\sigma _1}} {S_{{\sigma _2}{\sigma _1},{\tau _2}}} = {S_{{\sigma _1}{\sigma _2}{\sigma _1},{\tau _2}}}{S_{\sigma _1,\sigma _1}} .\]
Now the lemma follows. 
\end{proof}

\begin{lemma} \label{l4}
From relation $r_4 = \sigma_1 \sigma_2 \tau_1 \sigma_2^{-1} \sigma_1^{-1} \tau_2^{-1}$  follows 6 relations, applying which we can remove 2 generators:
$$
{S_{1,{\tau _2}}} = {S_{{\sigma _1}{\sigma _2},{\tau _1}}},~~~{S_{{\sigma _1}{\sigma _2}{\sigma _1},{\tau _1}}} = {S_{{\sigma _2},{\tau _2}}}$$

and we get relations:

$$
{S_{{\sigma _1},{\sigma _1}}}S_{{\sigma _2},{\tau _1}}=S_{{\sigma _1},{\tau _2}}{S_{{\sigma _1},{\sigma _1}}},
$$

$${S_{{\sigma _1},{\sigma _1}}}{S_{{\sigma _2\sigma _1},{\tau _1}}}=S_{{\sigma _1 \sigma _2},{\tau _2}}{S_{{\sigma _1},{\sigma _1}}},$$

$${S_{\sigma_2 \sigma_1,\sigma _1}S_{{\sigma _2},{\sigma _2}}}S_{1,\tau _1}= S_{{ \sigma _2  \sigma _1},{\tau _2}}S_{{ \sigma _2  \sigma _1},{\sigma _1}} S_{{ \sigma _2},{\sigma _2}},$$

$${S_{\sigma_2,\sigma _2} S_{{\sigma _1 \sigma _2},{\sigma _2}}}S_{\sigma_1,\tau _1}=S_{{\sigma _2 \sigma _1 },{\tau _2}}S_{{ \sigma _2 },{\sigma _2}}S_{{ \sigma _1\sigma _2},{\sigma _2}}.$$

\end{lemma}
\begin{proof}Take the relation $r_4 = \sigma_1 \sigma_2 \tau_1 \sigma_2^{-1} \sigma_1^{-1} \tau_2^{-1}$ and rewrite in the new generators:

\[{r_{4,1}} = {S_{{\sigma _1}{\sigma _2},{\tau _1}}}S_{1,{\tau _2}}^{ - 1} = 1\]
or
\[{S_{1,{\tau _2}}} = {S_{{\sigma _1}{\sigma _2},{\tau _1}}}.\]

Conjugating by $(\sigma_1)^{-1}$

\[{r_{4,{\sigma _1}}} = {S_{1,\sigma _1}S_{{\sigma _1},{\sigma _1}}}S_{1,\sigma _2}{S_{{\sigma _2},{\tau _1}}}S_{{1,\sigma _2}}^{ - 1}S_{{\sigma _1,\sigma _1}}^{ - 1}S_{{\sigma _1},{\tau _2}}^{ - 1}S_{{1,\sigma _1}}^{ - 1} ={S_{{\sigma _1},{\sigma _1}}} {S_{{\sigma _2},{\tau _1}}}S_{{\sigma _1,\sigma _1}}^{ - 1}S_{{\sigma _1},{\tau _2}}^{ - 1}= 1\]
or
$$
{S_{{\sigma _1},{\sigma _1}}}S_{{\sigma _2},{\tau _1}}=S_{{\sigma _1},{\tau _2}}{S_{{\sigma _1},{\sigma _1}}}.
$$

Next relation

\[{r_{4,{\sigma _2}}} = {S_{{\sigma _1}{\sigma _2}{\sigma _1},{\tau _1}}} S_{{\sigma _2},{\tau _2}}^{ - 1}  = 1\]
or
$${S_{{\sigma _1}{\sigma _2}{\sigma _1},{\tau _1}}} =S_{{\sigma _2},{\tau _2}}.$$

Next relation
\[{r_{4,{\sigma _1}{\sigma _2}}} = {S_{1,\sigma _1}S_{{\sigma _1},{\sigma _2}}}S_{\sigma _1 \sigma _2,\sigma _1}S_{\sigma _1 \sigma _2 \sigma _1,\sigma _2}{S_{{\sigma _2\sigma _1},{\tau _1}}}S_{{\sigma _1\sigma _2\sigma _1,\sigma _2}}^{ - 1}S_{{\sigma _1\sigma _2,\sigma _1}}^{ - 1}S_{{\sigma _1 \sigma _2},{\tau _2}}^{ - 1}S_{{\sigma _1,\sigma _2}}^{ - 1} S_{{1,\sigma _1}}^{ - 1} = \]
\[=S_{\sigma _1 \sigma _2 \sigma _1,\sigma _2}{S_{{\sigma _2\sigma _1},{\tau _1}}}S_{{\sigma _1\sigma _2\sigma _1,\sigma _2}}^{ - 1}S_{{\sigma _1 \sigma _2},{\tau _2}}^{ - 1}=1,\]

then $${S_{{\sigma _1},{\sigma _1}}}{S_{{\sigma _2\sigma _1},{\tau _1}}}=S_{{\sigma _1 \sigma _2},{\tau _2}}{S_{{\sigma _1},{\sigma _1}}}.$$

Take the relation:
\[{r_{4,{\sigma _2}{\sigma _1}}} = {S_{\sigma_2 \sigma_1,\sigma _1}S_{{\sigma _2},{\sigma _2}}}S_{1,\tau _1}S_{{ \sigma _2},{\sigma _2}}^{ - 1}S_{{ \sigma _2  \sigma _1},{\sigma _1}}^{ - 1}S_{{ \sigma _2  \sigma _1},{\tau _2}}^{ - 1}=1 \]
 
and we have the relation

$${S_{\sigma_2 \sigma_1,\sigma _1}S_{{\sigma _2},{\sigma _2}}}S_{1,\tau _1}= S_{{ \sigma _2  \sigma _1},{\tau _2}}S_{{ \sigma _2  \sigma _1},{\sigma _1}} S_{{ \sigma _2},{\sigma _2}}.$$
The relation:  $r_{4,{\sigma _1}{\sigma _2}{\sigma _1}} ={S_{\sigma_1 \sigma_2 \sigma_1,\sigma _1} S_{{\sigma _1 \sigma _2},{\sigma _2}}}S_{\sigma_1,\tau _1}S_{{ \sigma _1\sigma _2},{\sigma _2}}^{ - 1}S_{{ \sigma _1 \sigma _2  \sigma _1},{\sigma _1}}^{ - 1}S_{{\sigma _1 \sigma _2  \sigma _1},{\tau _2}}^{ - 1}=1 $
or
$${S_{\sigma_1 \sigma_2 \sigma_1,\sigma _1} S_{{\sigma _1 \sigma _2},{\sigma _2}}}S_{\sigma_1,\tau _1}=S_{{\sigma _1 \sigma _2  \sigma _1},{\tau _2}}S_{{ \sigma _1 \sigma _2  \sigma _1},{\sigma _1}}S_{{ \sigma _1\sigma _2},{\sigma _2}} .$$
Hence, we have proven the lemma. 
\end{proof}


\begin{lemma} \label{l5}
From relation $r_5 = \sigma_2 \sigma_1 \tau_2 \sigma_1^{-1} \sigma_2^{-1} \tau_1^{-1}$ follows  relations:

$$S_{{\sigma _2}{\sigma _1},{\tau _2}}= S_{{1},{\tau _1}},$$
$$S_{{\sigma _1}{\sigma _2}{\sigma _1},{\tau _2}}= S_{{\sigma _1},{\tau _1}},$$
$${S_{{\sigma _2},{\sigma _2}}} {S_{\sigma_1,{\tau _2}}}= S_{{\sigma _2},{\tau _1}} S_{{\sigma _2}, {\sigma _2}},$$
$$S_{{\sigma _1 \sigma _2}, {\sigma _2}}  {S_{{\sigma _1},{\sigma _1}}} S_{{\sigma_2, \tau _2}}  =  S_{{\sigma _2, \tau _2}}  S_{{\sigma _1 \sigma _2}, {\sigma _2}} S_{{\sigma _1}, {\sigma _1}},$$
$$S_{{\sigma _2}, {\sigma _2}} S_{{\sigma_1, \tau _2}}  =  S_{{\sigma _2}, {\tau _1}} S_{{\sigma _2}, {\sigma _2}},$$
$$S_{{\sigma _1}, {\sigma _1}}S_{{\sigma _2 \sigma _1 }, {\sigma _1}}  S_{{\sigma_2, \tau _2}}  =  S_{{\sigma _2}, {\tau _2}} S_{{ \sigma _1 }, {\sigma _1}} S_{{\sigma _2 \sigma _1}, {\sigma _1}}.$$

\end{lemma}
\begin{proof} Take the relation $r_5 = \sigma_2 \sigma_1 \tau_2 \sigma_1^{-1} \sigma_2^{-1} \tau_1^{-1}$ and rewrite in the new generators:

\[{r_{5,1}} = S_{{\sigma _2}{\sigma _1},{\tau _2}} S_{{1},{\tau _1}}^{ - 1} = 1\]
we get relation
$$
S_{{\sigma _2}{\sigma _1},{\tau _2}}= S_{{1},{\tau _1}}.
$$

Relation

\[{r_{5,{\sigma _1}}} = S_{{\sigma _1}{\sigma _2}{\sigma _1},{\tau _2}}
S_{{\sigma _1},{\tau _1}}^{ - 1} = 1\]
or
 $$S_{{\sigma _1}{\sigma _2}{\sigma _1},{\tau _2}}= S_{{\sigma _1},{\tau _1}}.$$
Relation

\[{r_{5,{\sigma _2}}} = {S_{{\sigma _2},{\sigma _2}}} {S_{\sigma_1,{\tau _2}}} S_{{\sigma _2} {\sigma _2}}^{ - 1} S_{{\sigma _2}, {\tau_1}}^{ - 1}=1\]

we get relation
$$
 {S_{{\sigma _2},{\sigma _2}}} {S_{\sigma_1,{\tau _2}}}= S_{{\sigma _2},{\tau _1}} S_{{\sigma _2}, {\sigma _2}}.
$$

From relation ${r_{5,{\sigma _1}{\sigma _2}}} =  1$ follows relation

$$S_{{\sigma _1 \sigma _2}, {\sigma _2}}  {S_{{\sigma _1},{\sigma _1}}}S_{{1, \tau _2}}  =  S_{{\sigma _1 \sigma _2}, {\tau _1}}  S_{{\sigma _1 \sigma _2}, {\sigma _2}} S_{{\sigma _1}, {\sigma _1}},$$

$$S_{{\sigma _1 \sigma _2}, {\sigma _2}}  {S_{{\sigma _1},{\sigma _1}}} S_{{\sigma_2, \tau _2}}  =  S_{{\sigma _2, \tau _2}}  S_{{\sigma _1 \sigma _2}, {\sigma _2}} S_{{\sigma _1}, {\sigma _1}}.$$

From $r_{5,{\sigma _2}{\sigma _1}} = 1$ follows relation

$$S_{{\sigma _1 \sigma _2 \sigma _1 }, {\sigma _1}} S_{{\sigma _1 \sigma _2, \tau _2}}  =  S_{{\sigma _2 \sigma _1}, {\tau _1}}  S_{{\sigma _1 \sigma _2 \sigma _1 }, {\sigma _1}},$$

$$S_{{\sigma _2}, {\sigma _2}} S_{{\sigma_1, \tau _2}}  =  S_{{\sigma _2}, {\tau _1}} S_{{\sigma _2}, {\sigma _2}}.$$

From ${r_{5,{\sigma _1}{\sigma _2}{\sigma _1}}} =  1$ follows relation

$$S_{{\sigma _1 \sigma _2 \sigma _1 }, {\sigma _2}} S_{{\sigma _2 \sigma _1 }, {\sigma _1}} S_{{\sigma _2, \tau _2}}  =  S_{{\sigma _1 \sigma _2 \sigma _1 }, {\tau _1}} S_{{\sigma _1 \sigma _2 \sigma _1 }, {\sigma _2}} S_{{\sigma _2 \sigma _1 }, {\sigma _1}},$$

$$S_{{\sigma _1}, {\sigma _1}}S_{{\sigma _2 \sigma _1 }, {\sigma _1}}  S_{{\sigma_2, \tau _2}}  =  S_{{\sigma _2}, {\tau _2}} S_{{ \sigma _1 }, {\sigma _1}} S_{{\sigma _2 \sigma _1}, {\sigma _1}}.$$
Hence we have proven the lemma. \end{proof}

\medskip

\medskip

Therefore,

$$S_{\sigma_1,\tau_1} = S_{1,\tau_1}; ~
S_{\sigma_2\sigma_1, \tau_1} = S_{\sigma_2, \tau_1}; ~
S_{\sigma_1 \sigma_2 \sigma_1,\tau_1} =S_{\sigma_1 \sigma_2, \tau_1};$$
$$S_{\sigma_2 \sigma_1, \sigma_2} = 1; ~
S_{\sigma_1, \sigma_1}=S_{\sigma_1 \sigma_2 \sigma_1, \sigma_2}; ~
S_{\sigma_1 \sigma_2 \sigma_1, \sigma_1} = S_{\sigma_2, \sigma_2}; 
$$
$$
{S_{1,{\tau _2}}} = {S_{{\sigma _2},{\tau _2}}}; ~
S_{{\sigma _1},{\tau _2}} = {S_{{\sigma _1}{\sigma _2},{\tau _2}}}; ~
{S_{{\sigma _2}{\sigma _1},{\tau _2}}} = {S_{{\sigma _1}{\sigma _2}{\sigma _1},{\tau _2}}};$$
$${S_{1,{\tau _2}}} = {S_{{\sigma _1}{\sigma _2},{\tau _1}}}; ~
{S_{{\sigma _1}{\sigma _2}{\sigma _1},{\tau _1}}} = {S_{{\sigma _2},{\tau _2}}}; ~
S_{{\sigma _2}{\sigma _1},{\tau _2}}= S_{{1},{\tau _1}}; ~
S_{{\sigma _1}{\sigma _2}{\sigma _1},{\tau _2}}= {S_{{\sigma _1},{\tau _1}}};$$

$$S_{\sigma_1,\sigma_1} S_{\sigma_1,\tau_1} =S_{\sigma_1, \tau_1}S_{\sigma_1, \sigma_1}; $$
$$S_{\sigma_2 \sigma_1, \sigma_1}S_{\sigma _2,\tau _1}= S_{\sigma_2, \tau_1}S_{\sigma _2\sigma _1, \sigma _1};$$
$$S_{ \sigma_2, \sigma_2}S_{ \sigma_2,\tau_2} = S_{\sigma_2, \tau_2} S_{\sigma_2, \sigma_2};$$
$${S_{{\sigma _1},{\sigma _1}}}{S_{\sigma _2\sigma _1,\sigma _1}} = {S_{\sigma _1\sigma _2,\sigma _2}}{S_{\sigma _1,\sigma _1}};$$
$${S_{\sigma _2\sigma _1,\sigma _1}}{S_{\sigma _2,\sigma _2}} = {S_{\sigma _2,\sigma _2}} {S_{\sigma _1 \sigma _2,\sigma _2}};$$
$${S_{{\sigma _2},{\sigma _2}}}{S_{{\sigma _1}{\sigma _2},{\sigma _2}}}{S_{\sigma _1 \sigma _1}}= {S_{{\sigma _1},{\sigma _1}}}{S_{{\sigma _2}{\sigma _1},{\sigma _1}}}{S_{{\sigma _2},{\sigma _2}}};$$
\[{S_{\sigma _1 \sigma _2, \sigma _2}{S_{{\sigma _1}, {\tau _2}}}= {S_{{\sigma _1}, {\tau _2}}}S_{{\sigma _1}{\sigma _2},{\sigma _2}}};\]
$$
{S_{{\sigma _1},{\sigma _1}}}S_{{\sigma _2},{\tau _1}}=S_{{\sigma _1},{\tau _2}}{S_{{\sigma _1},{\sigma _1}}};
$$
$${S_{\sigma_2 \sigma_1,\sigma _1}S_{{\sigma _2},{\sigma _2}}}S_{\sigma_1,\tau _1}= S_{{\sigma _1},{\tau _1}}S_{{ \sigma _2  \sigma _1},{\sigma _1}} S_{{ \sigma _2},{\sigma _2}};$$
$${S_{\sigma_2,\sigma _2} S_{{\sigma _1 \sigma _2},{\sigma _2}}}S_{\sigma_1,\tau _1}=S_{{\sigma _1 },{\tau _1}}S_{{ \sigma _2 },{\sigma _2}}S_{{ \sigma _1\sigma _2},{\sigma _2}};$$
$${S_{{\sigma _2},{\sigma _2}}} {S_{\sigma_1,{\tau _2}}}= S_{{\sigma _2},{\tau _1}} S_{{\sigma _2}, {\sigma _2}};$$
$$S_{{\sigma _1 \sigma _2}, {\sigma _2}}  {S_{{\sigma _1},{\sigma _1}}} S_{{\sigma_2, \tau _2}}  =  S_{{\sigma _2, \tau _2}}  S_{{\sigma _1 \sigma _2}, {\sigma _2}} S_{{\sigma _1}, {\sigma _1}};$$
$$S_{{\sigma _2}, {\sigma _2}} S_{{\sigma_1, \tau _2}}  =  S_{{\sigma _2}, {\tau _1}} S_{{\sigma _2}, {\sigma _2}};$$
$$S_{{\sigma _1}, {\sigma _1}}S_{{\sigma _2 \sigma _1 }, {\sigma _1}}  S_{{\sigma_2, \tau _2}}  =  S_{{\sigma _2}, {\tau _2}} S_{{ \sigma _1 }, {\sigma _1}} S_{{\sigma _2 \sigma _1}, {\sigma _1}}.$$

\begin{lemma} The following equalities hold
\begin{align*}
& {S_{1,{\tau _1}}} = {\tau _1} = {c_{12}},\\
& {S_{1,{\tau _2}}} = {\tau _2} = {c_{23}},\\
& {S_{{\sigma _1},{\sigma _1}}} = \sigma _1^2 = {a_{12}},\\
&{S_{{\sigma _1},{\tau _1}}} = {\tau _1} = {c_{12}},\\
&{S_{{\sigma _1},{\tau _2}}} = {\sigma _1}{\tau _2}\sigma _1^{ - 1} = a_{12} c_{13} a_{12}^{ - 1},\\
&{S_{{\sigma _2},{\sigma _2}}} = \sigma _2^2 = {a_{23}}, \\
& {S_{{\sigma _2},{\tau _1}}} = {\sigma _2}{\tau _1}\sigma _2^{ - 1} = {c_{13}},\\
& {S_{{\sigma _2},{\tau _2}}} = {\tau _2} = {c_{23}},\\
& {S_{{\sigma _1}{\sigma _2},{\sigma _2}}} = {\sigma _1}\sigma _2^2\sigma _1^{ - 1} = {a_{12} a_{13}a_{12}^{-1}},\\
& {S_{{\sigma _1}{\sigma _2},{\tau _1}}} = {\sigma _1}{\sigma _2}{\tau _1}\sigma _2^{ - 1}\sigma _1^{ - 1} = {c_{23}},\\
&{S_{{\sigma _1}{\sigma _2},{\tau _2}}} = {\sigma _1}{\tau _2}\sigma _1^{ - 1} = {a_{12}{c_{13}}a_{12}^{-1}},\\
& {S_{{\sigma _2}{\sigma _1},{\sigma _1}}} = {\sigma _2}\sigma _1^2\sigma _2^{ - 1} = {a_{13}},\\
& {S_{{\sigma _2}{\sigma _1},{\tau _1}}} = {\sigma _2}{\tau _1}\sigma _2^{ - 1} = {c_{13}},\\
& {S_{{\sigma _2}{\sigma _1},{\tau _2}}} = {\sigma _2}{\sigma _1}{\tau _2}\sigma _1^{ - 1}\sigma _2^{ - 1} = {c_{12}},\\
&{S_{{\sigma _1}{\sigma _2}{\sigma _1},{\sigma _1}}} = {\sigma _1}{\sigma _2}\sigma _1^2\sigma _2^{ - 1}\sigma _1^{ - 1} = {a_{23}},\\
&{S_{{\sigma _1}{\sigma _2}{\sigma _1},{\sigma _2}}} = {\sigma _1}{\sigma _2}\sigma _1 \sigma_2 \sigma _1^{ - 1}\sigma _2^{ - 1} = {a_{12}},\\
&{S_{{\sigma _1}{\sigma _2}{\sigma _1},{\tau _1}}} = {\sigma _1}{\sigma _2}{\tau _1}\sigma _2^{ - 1}\sigma _1^{ - 1} = {c_{23}},\\
& {S_{{\sigma _1}{\sigma _2}{\sigma _1},{\tau _2}}} = {\sigma _1}{\sigma _2}{\sigma _1}{\tau _2}\sigma _1^{ - 1}\sigma _2^{ - 1}\sigma _1^{ - 1} = {c_{12}}.\\
\end{align*}
Thus $ST_3$ is generated by elements $a_{12}$, $a_{13}$, $a_{23}$, $c_{12}$, $c_{13}$, $c_{23}$. Their geometric interpretations are given in Fig.~\ref{relaij}.

\begin{figure}[h]
\centering{
\includegraphics[width=10cm]{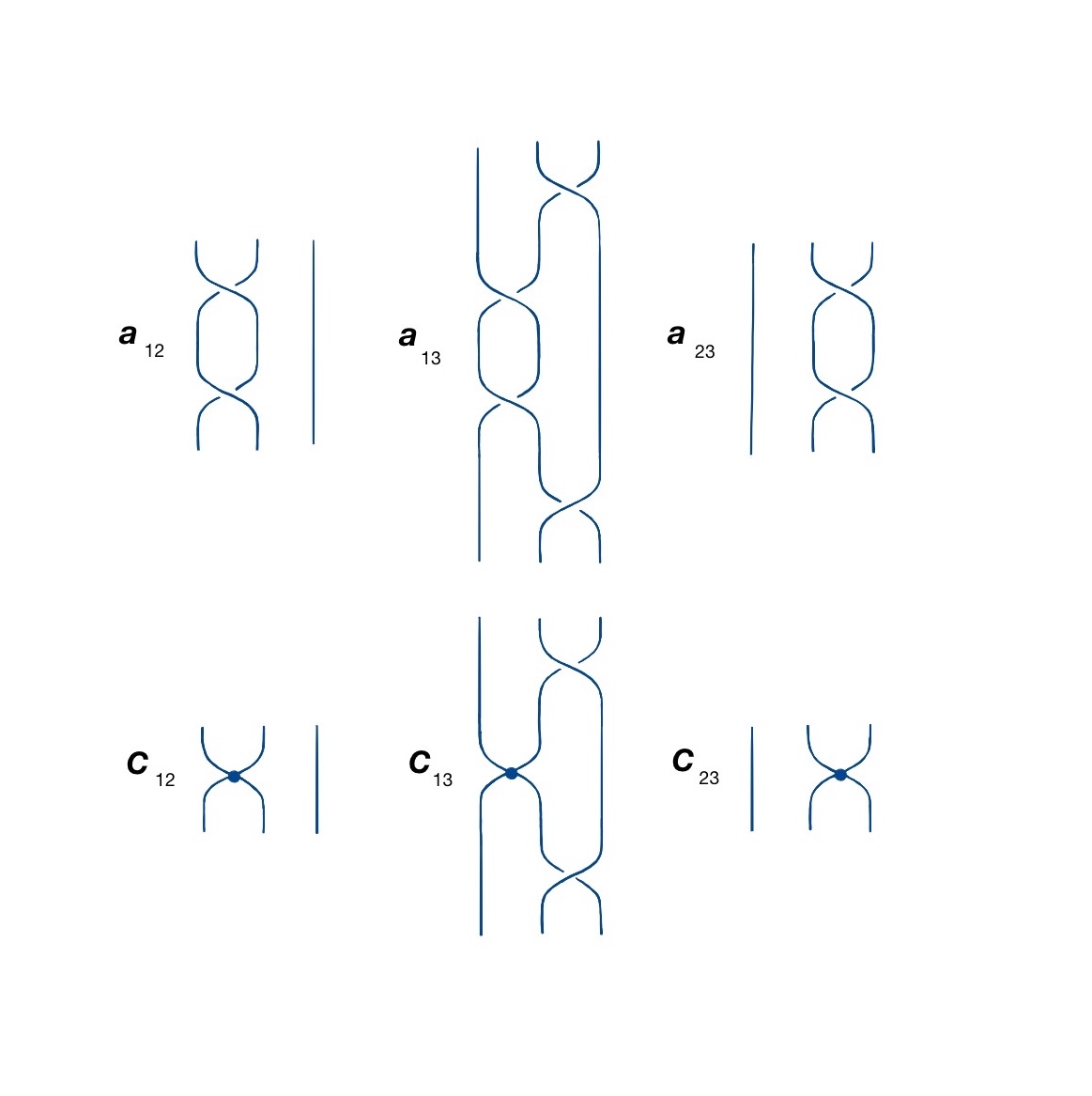}
\caption{Geometric interpretation of generators of $ST_3$} \label{relaij}
}
\end{figure}

\end{lemma}

\begin{proposition} \label{p4.1}
Generators of $SB_3$ act on the generators of $ST_3$ by the rules:

-- action of $\sigma_1^{-1}$:

$$
a_{12}^{\sigma_1^{-1}} = a_{12},~~~a_{13}^{\sigma_1^{-1}} =  a_{23},~~~ a_{23}^{\sigma_1^{-1}} = a_{12}a_{13}a_{12}^{-1},
$$
$$
c_{12}^{\sigma_1^{-1}} = c_{12},~~~c_{13}^{\sigma _1^{-1}} = c_{23},~~~ c_{23}^{\sigma _1^{-1}} = a_{12}c_{13}a_{12}^{-1},
$$

-- action of $\sigma_2^{-1}$:

$$
a_{12}^{\sigma_2^{-1}} = a_{13},~~~a_{13}^{\sigma_2^{-1}} = a_{23}a_{12}a_{23}^{-1},~~~ a_{23}^{\sigma_2^{-1}} = a_{23},
$$
$$
c_{12}^{\sigma _2^{-1}} = c_{13},~~~c_{13}^{\sigma _2^{-1}} = a_{23}{c_{12}}c_{23}^{-1},~~~ c_{23}^{\sigma_2^{-1}} = c_{23},
$$

-- action of $\tau_1^{-1}$:

$$
a_{12}^{\tau_1^{-1}} = c_{12}a_{12}c_{12}^{-1} ,~~~a_{13}^{\tau_1^{-1}} = c_{12} a_{13} c_{12}^{-1},~~~a_{23}^{\tau_1^{-1}} = c_{12} a_{23}c_{12}^{-1},
$$
$$
c_{12}^{\tau_1^{-1}} = c_{12},~~~c_{13}^{\tau _1^{-1}} = c_{12}{c_{13}}{c_{12}^{-1}},~~~c_{23}^{{\tau _1^{-1}}} = c_{12}{c_{23}}{c_{12}^{-1}},
$$

-- action of $\tau_2^{-1}$:

$$
a_{12}^{\tau _2^{-1}} = c_{23}{a_{12}}{c_{23}^{-1}},~~~
a_{13}^{\tau _2^{-1}} = c_{23}{a_{13}}c_{23}^{-1},~~~a_{23}^{\tau _2^{-1}} = c_{23} a_{23}c_{23}^{-1},
$$

$$
c_{12}^{\tau _2^{-1}} = c_{23}{c_{12}}{c_{23}^{-1}},~~~c_{13}^{\tau _2^{-1}} =  c_{23} c_{13} c_{23}^{-1},~~~
c_{23}^{\tau_2^{-1}} = c_{23},
$$

-- action of $\sigma_1$:

$$
a_{12}^{\sigma_1} = a_{12},~~~a_{13}^{\sigma_1} = a_{13} a_{23} a_{13}^{-1},~~~ a_{23}^{\sigma_1} = a_{13},
$$
$$
c_{12}^{\sigma_1} = c_{12},~~~c_{13}^{{\sigma _1}} = a_{12}^{ - 1}{c_{23}}a_{12},~~~ c_{23}^{\sigma _1} = c_{13},
$$

-- action of $\sigma_2$:

$$
a_{12}^{\sigma_2} = a_{23}^{-1} a_{13} a_{23},~~~a_{13}^{\sigma_2} = a_{12},~~~ a_{23}^{\sigma_2} = a_{23},
$$
$$
c_{12}^{{\sigma _2}} = a_{12}{c_{13}}a_{12}^{ - 1},~~~c_{13}^{{\sigma _2}} = {c_{12}},~~~ c_{23}^{\sigma_2} = c_{23},
$$

-- action of $\tau_1$:

$$
a_{12}^{\tau_1} = a_{12},~~~a_{13}^{\tau_1} = c_{12}^{-1} a_{13} c_{12},~~~a_{23}^{\tau_1} = c_{12}^{-1} a_{23}c_{12},
$$
$$
c_{12}^{\tau_1} = c_{12},~~~c_{13}^{{\tau _1}} = c_{12}^{ - 1}{c_{13}}{c_{12}},~~~c_{23}^{{\tau _1}} = c_{12}^{ - 1}{c_{23}}{c_{12}},
$$

-- action of $\tau_2$:

$$
a_{12}^{\tau _2} = c_{23}^{ - 1}{a_{12}}{c_{23}},~~~
a_{13}^{\tau _2} = c_{23}^{ - 1}{a_{13}}c_{23},~~~a_{23}^{\tau _2} = c_{23}^{ - 1}a_{23}c_{23},
$$

$$
c_{12}^{\tau _2} = c_{23}^{ - 1}{c_{12}}{c_{23}},~~~c_{13}^{\tau _2} =  c_{23}^{-1} c_{13} c_{23},~~~
c_{23}^{\tau_2} = c_{23}.
$$

\end{proposition}

\begin{proof}

Let us prove some formulas:

$$
\tau_2 c_{13} \tau_2^{-1} = \tau_2  \sigma_2  \tau_1 \sigma_2^{-1} \tau_2^{-1}= S_{1, \tau_2}
S_{\sigma_2, \tau_1} S_{1, \tau_2}^{-1} = c_{23}  c_{13} c_{23}^{-1},
$$

$$
\tau_1^{-1} a_{23} \tau_1 = \tau_1^{-1} \sigma_2 \sigma_2  \tau_1 = S_{1, \tau_1}^{-1}
S_{\sigma_2, \sigma_2} S_{1, \tau_1} = c_{12}^{-1} a_{23} c_{12},
$$

$$
\tau_1^{-1} c_{23} \tau_1 = \tau_1^{-1} \tau_2 \tau_1 = S_{1, \tau_1}^{-1}
S_{1, \tau_2}S_{1, \tau_1} = c_{12}^{-1} c_{23} c_{12},
$$

$$
\tau_1^{-1} a_{13} \tau_1 = \tau_1^{-1} \sigma_2 \sigma_1 \sigma_1 \sigma_2^{-1} \tau_1 = S_{1, \tau_1}^{-1}
S_{\sigma_2 \sigma_1, \sigma_1}S_{1, \tau_1} = c_{12}^{-1} a_{13} c_{12},
$$

$$
  \sigma _2{a_{13}}{\sigma_2^{ - 1}} = \sigma _2 {\sigma _2}{\sigma _1}{\sigma _1}{\sigma _2^{ - 1} }\sigma _2^{ - 1} = 
  S_{{\sigma _2},{\sigma _2}}S_{{\sigma _1},{\sigma _1}}S_{{\sigma _2},{\sigma _2}}^{ - 1} = a_{23} {a_{12}}a_{23}^{ - 1}, 
$$

$$
\sigma _1^{ - 1}c_{13}{\sigma _1} = \sigma _1^{ - 1}{\sigma _2}{\tau _1}{\sigma _2}^{ - 1} {\sigma _1}= S_{{\sigma _1},{\sigma _1}}^{ - 1}{S_{{\sigma _1}{\sigma _2},{\tau _1}}}{S_{{\sigma _1},{\sigma _1}}} = a_{12}^{ - 1}{c_{23}}{a_{12}},
$$

$$
 \sigma _1 {c_{13}}{\sigma _1^{ - 1} } = \sigma _1{\sigma _2}{\tau _1}\sigma _2^{ - 1}{\sigma _1^{ - 1}} = S_{{\sigma _1}{\sigma _2},{\tau _1}}= c_{23}.
$$

\end{proof}

The group $ST_3$ has the following presentation.

\begin{theorem} \label{t1}
The group $ST_3$ is generated by elements
$$a_{12},~~a_{13},~~a_{23},~~c_{12},~~c_{13},~~c_{23},$$
subject to the defining  relations: 
\begin{equation} a_{12} c_{12} =c_{12} a_{12}   
\,\,\,\,\,\,(see \,\,  Fig.~\ref{rel1}), \end{equation}
\begin{figure}[h]
\centering{
\includegraphics[totalheight=4.5cm]{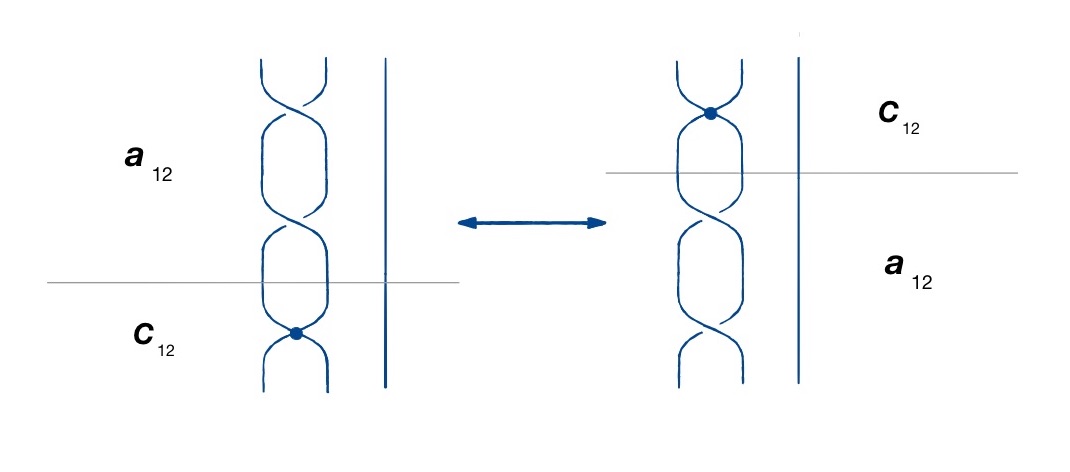}\\
\caption{Defining relation for $ST_3$: $a_{12} c_{12} =c_{12} a_{12} $} \label{rel1}
}
\end{figure}

\begin{equation} \label{eq2} a_{13} c_{13}= c_{13} a_{13} 
\,\,\,\,\,\,(see \,\,  Fig.~\ref{rel2}), \end{equation}
\begin{figure}[h]
\centering{
\includegraphics[totalheight=8.cm]{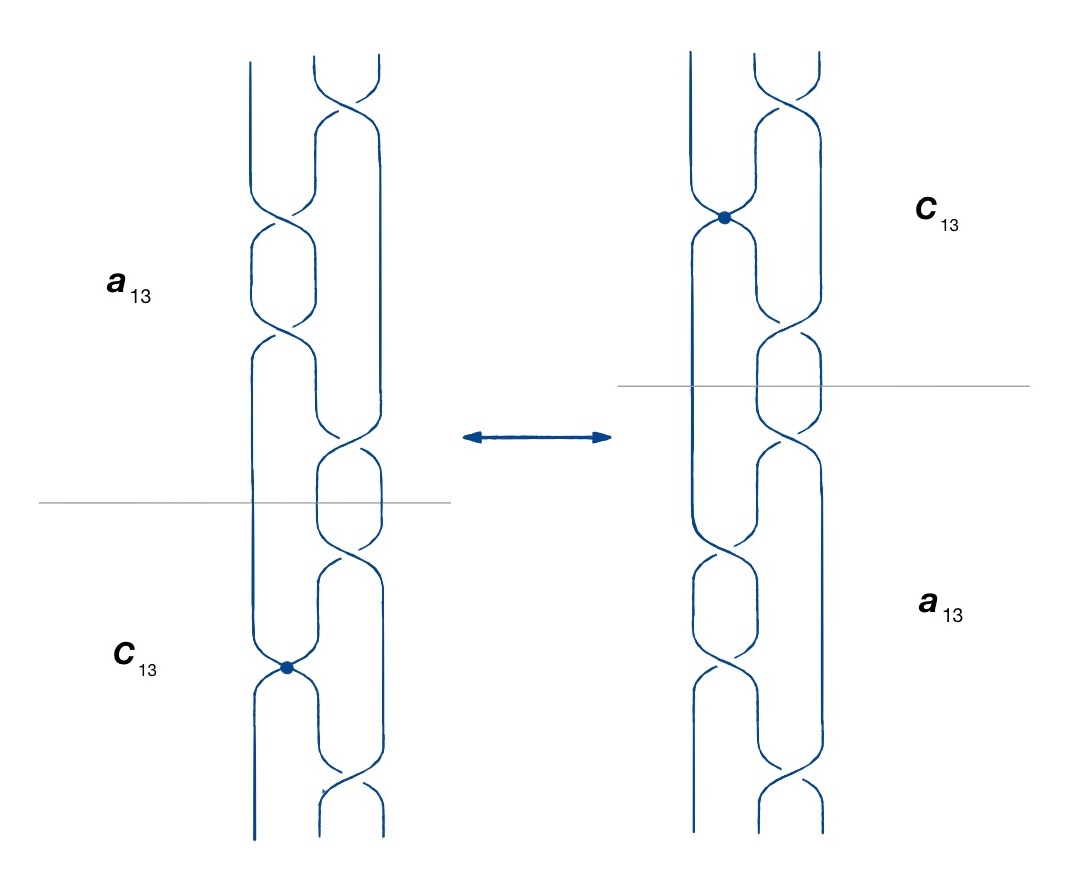}
\caption{Defining relation for $ST_3$: $a_{13} c_{13}= c_{13} a_{13} $} \label{rel2}
}
\end{figure}

\begin{equation} a_{23} c_{23}= c_{23} a_{23} 
\,\,\,\,\,\,(see \,\,  Fig.~\ref{rel3}), \end{equation}
\begin{figure}[h]
\centering{
\includegraphics[totalheight=4.5cm]{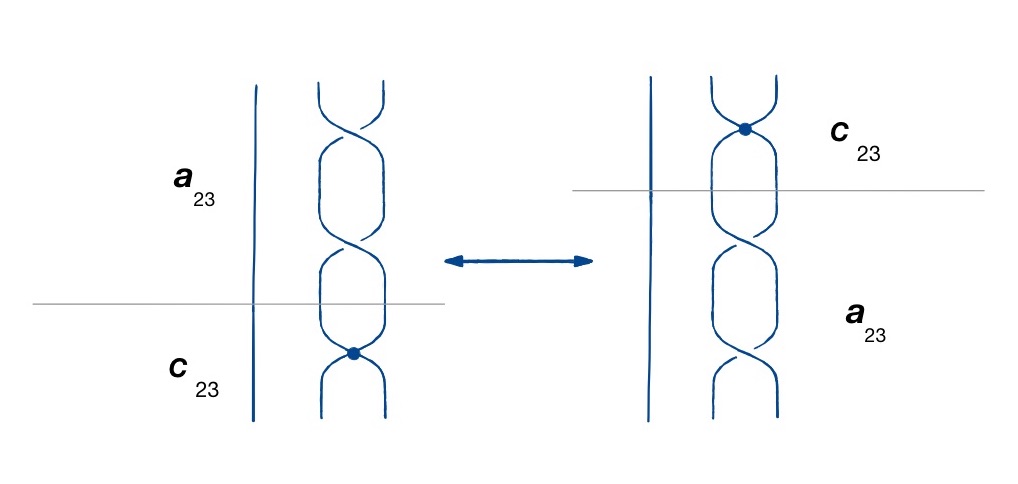}
\caption{Defining relation for $ST_3$: $a_{23} c_{23}= c_{23} a_{23}$} \label{rel3}
}
\end{figure}

\begin{equation} a_{12} a_{13} a_{12}^{-1}= a_{23}^{-1} a_{13} a_{23} 
\,\,\,\,\,\,(see \,\,  Fig.~\ref{rel4}), \end{equation}
\begin{figure}[h]
\centering{
\includegraphics[totalheight=9.cm]{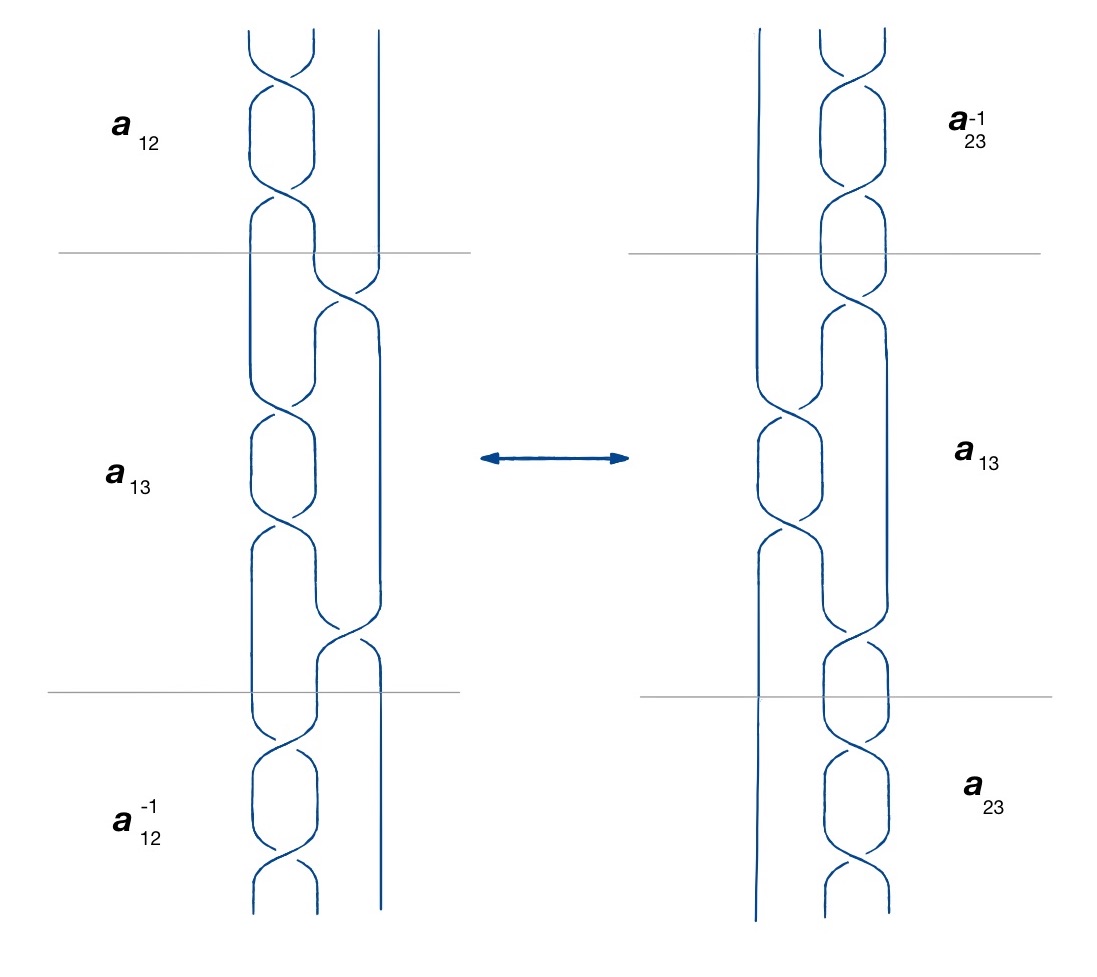}
\caption{Defining relation for $ST_3$: $ a_{12} a_{13} a_{12}^{-1}= a_{23}^{-1} a_{13} a_{23}$} \label{rel4}
}
\end{figure}

\begin{figure}[h]
\centering{
\includegraphics[totalheight=7.cm]{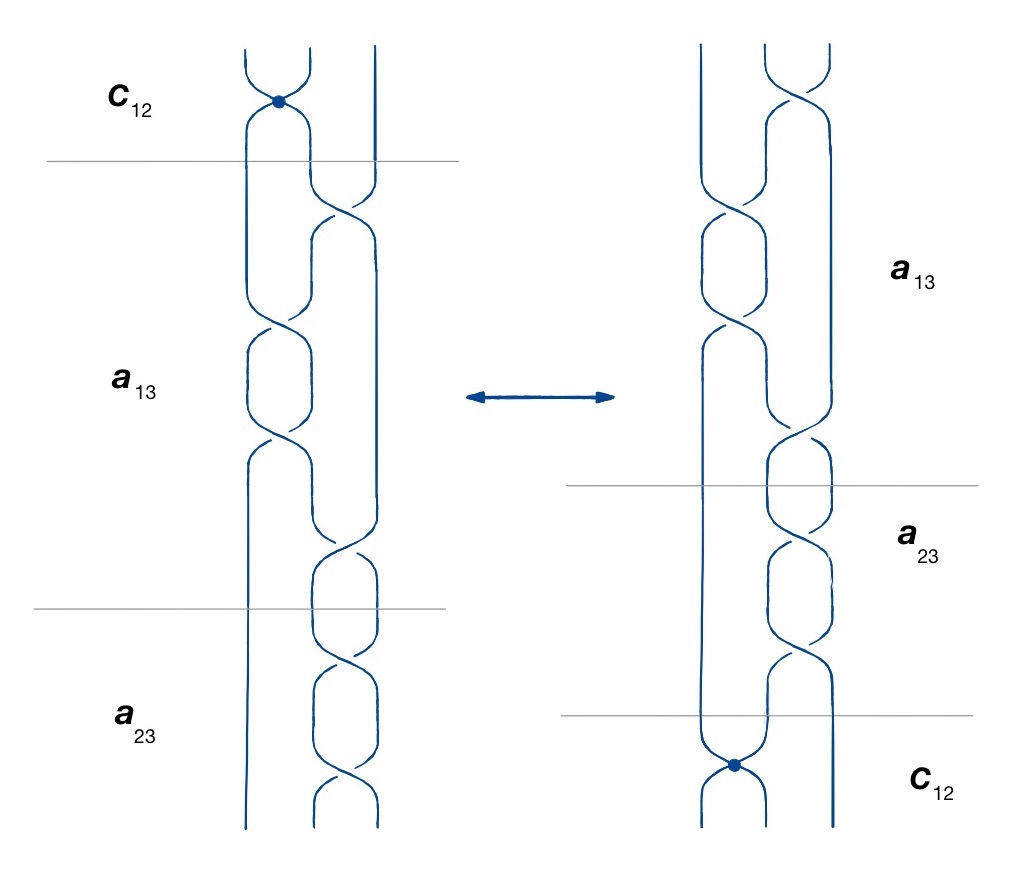}
\caption{Defining relation for $ST_3$: $c_{12} a_{13}a_{23} c_{12}^{-1}=a_{13} a_{23}$} \label{rel6}
}
\end{figure}

\begin{equation} c_{12} a_{13}a_{23} c_{12}^{-1}=a_{13} a_{23}
\,\,\,\,\,\,(see \,\,  Fig.~\ref{rel6}), \end{equation}

\begin{equation}  \label{r7} a_{12}c_{13} a_{12}^{-1} =a_{23}^{-1} c_{13} a_{23} 
\,\,\,\,\,\,(see \,\,  Fig.~\ref{rel7}),\end{equation}
\begin{figure}[h]
\centering{
\includegraphics[totalheight=7.cm]{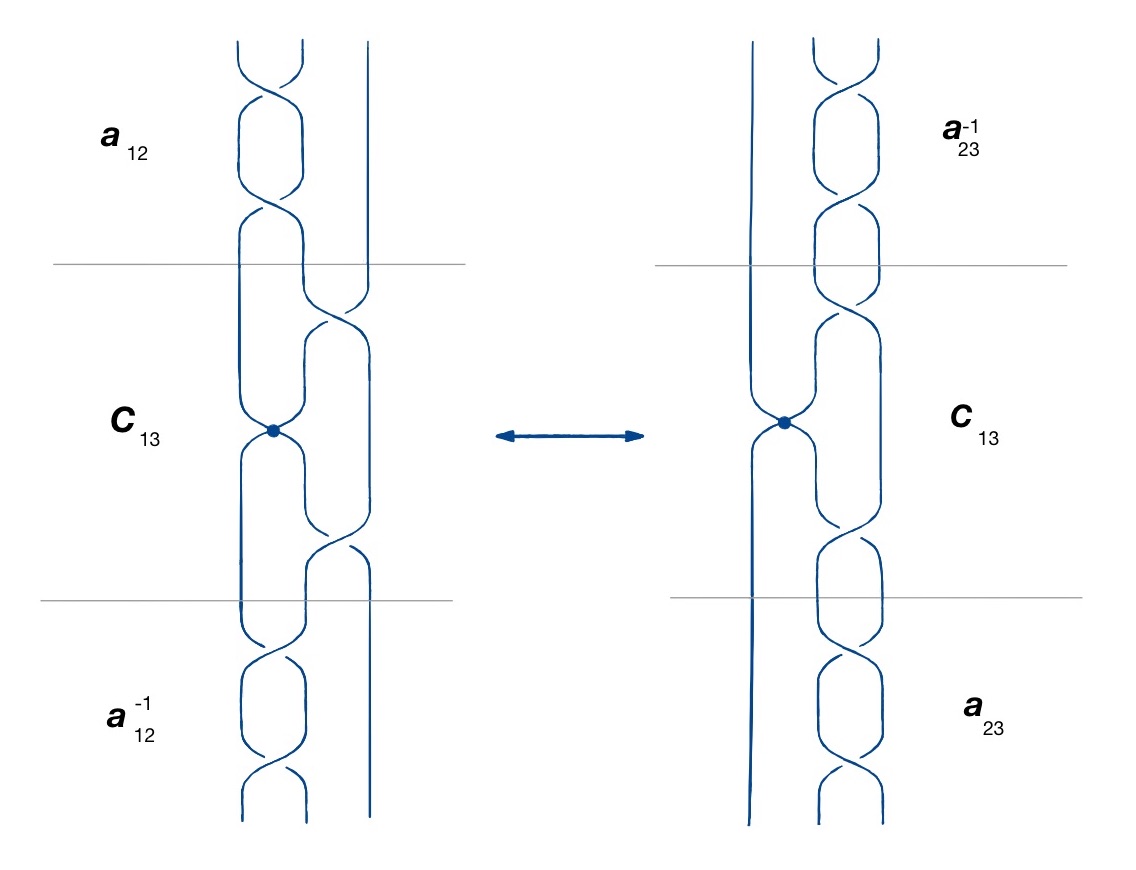}
\caption{Defining relation for $ST_3$: $a_{12}c_{13} a_{12}^{-1} =a_{23}^{-1} c_{13} a_{23}$} \label{rel7}
}
\end{figure}

\begin{equation} a_{12}^{-1} a_{23}a_{12}= a_{13} a_{23}a_{13}^{-1} 
\,\,\,\,\,\,(see \,\,  Fig.~\ref{rel5}), \end{equation}
\begin{figure}[h]
\centering{
\includegraphics[totalheight=10.cm]{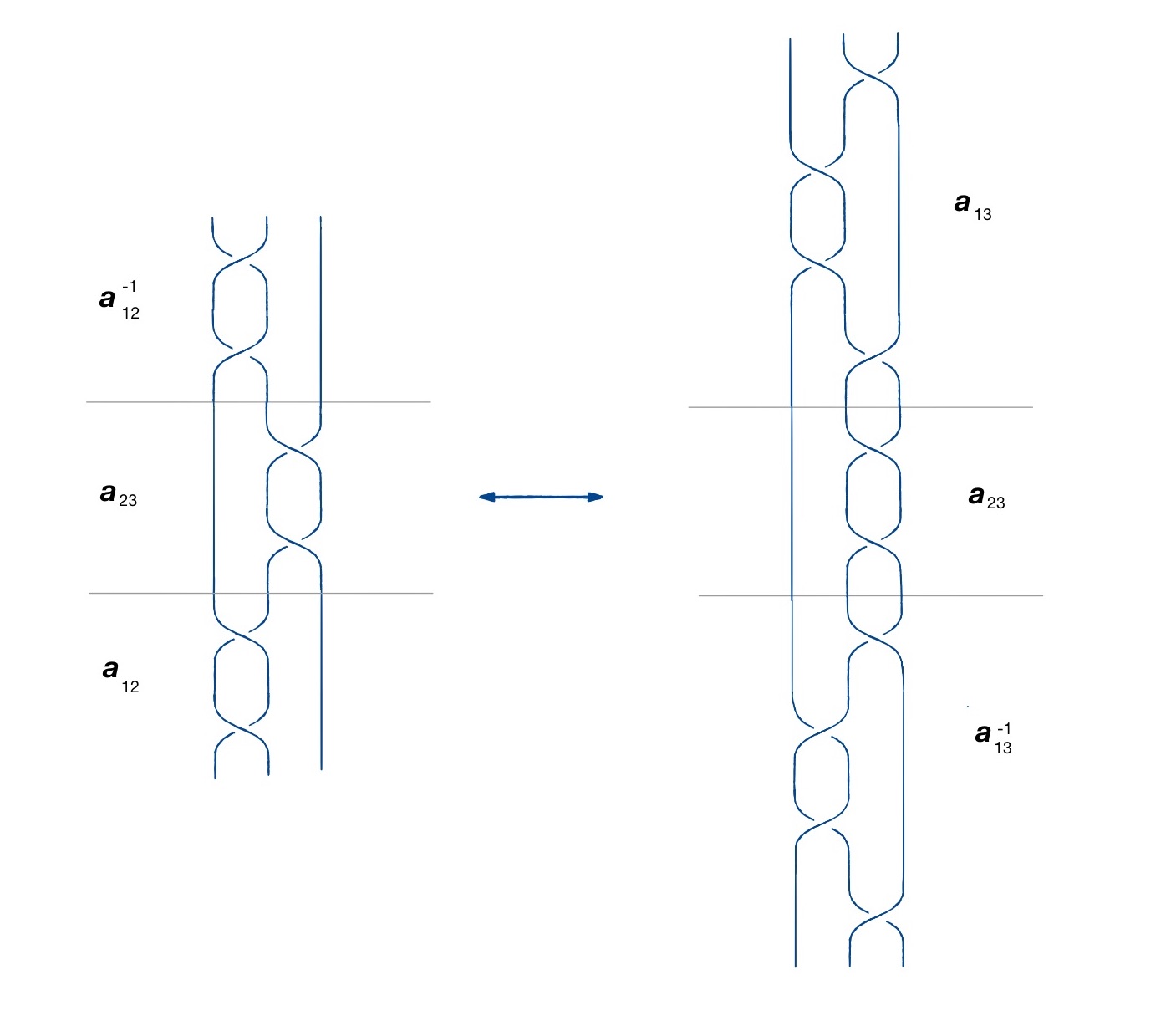}
\caption{Defining relation for $ST_3$: $ a_{12}^{-1} a_{23}a_{12}= a_{13} a_{23}a_{13}^{-1} $} \label{rel5}
}
\end{figure}

\begin{equation} a_{12}^{-1} c_{23} a_{12} =a_{13} c_{23} a_{13}^{-1}
\,\,\,\,\,\,(see \,\,  Fig.~\ref{rel8}).\end{equation}
\begin{figure}[h]
\centering{
\includegraphics[totalheight=9.cm]{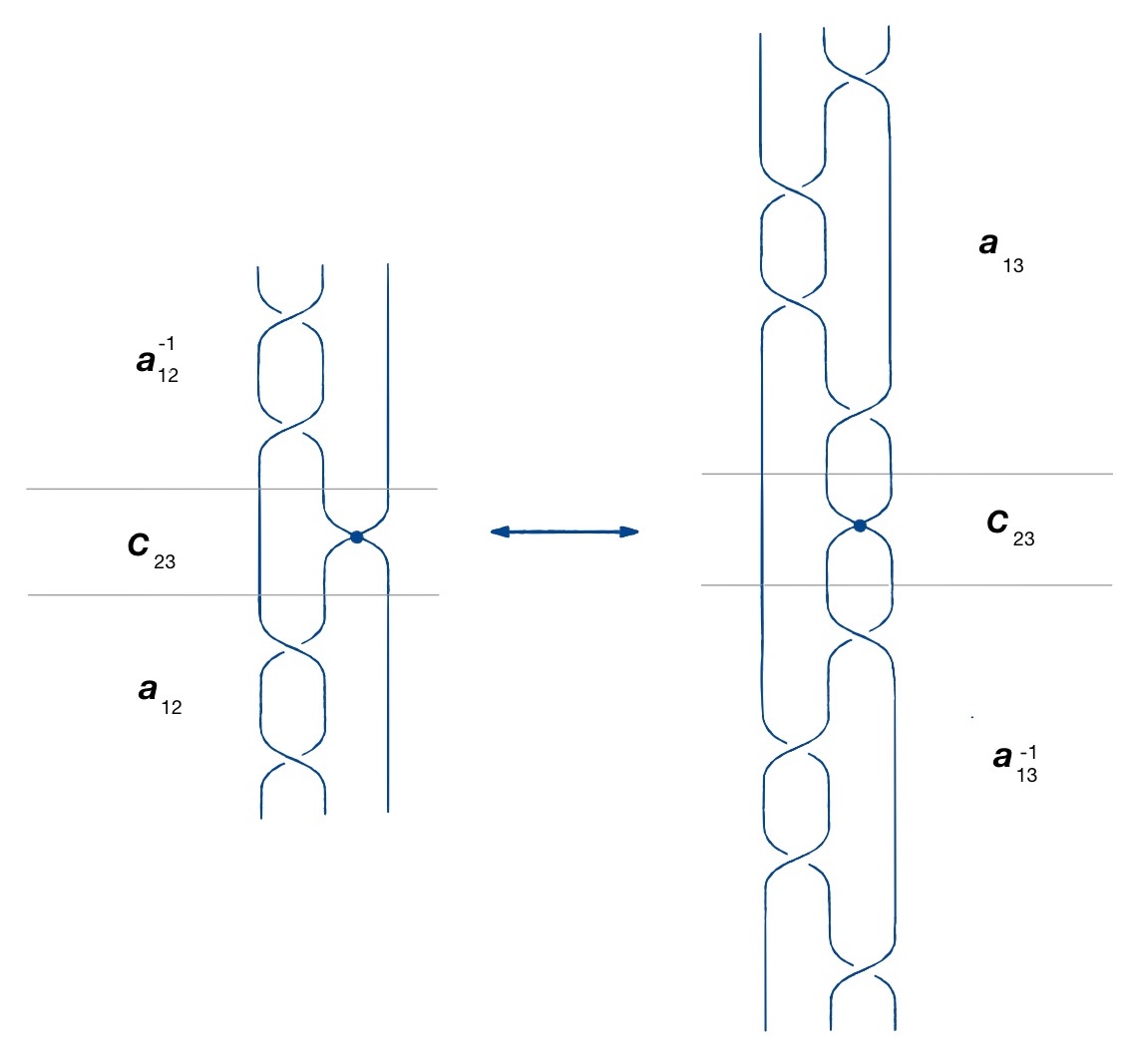}
\caption{Defining relation for $ST_3$: $a_{12}^{-1} c_{23} a_{12} =a_{13} c_{23} a_{13}^{-1}$} \label{rel8}
}
\end{figure}

\end{theorem}

\section{Structure of $ST_3$} \label{pd}

Some decomposition of $ST_3$ gives the following

\begin{theorem}  \label{t4.1} The group $ST_3$ is the semi-direct product of the normal subgroup
$$H=\langle a_{13} , a_{23} , c_{12}, c_{13}, c_{23} \ |  \ a_{13} c_{13}= c_{13} a_{13}, \ a_{23} c_{23}= c_{23} a_{23}, \ c_{12} a_{13}a_{23} c_{12}^{-1}=a_{13} a_{23}\rangle.$$ 
and the infinite cyclic group
 $U_2=\langle a_{12} \rangle$.

The group $H$ is an HNN extension of 
$$\Z^2 \ast \Z^2 \simeq \langle a_{13}, c_{23}, c_{13}, c_{23} \ | \ a_{13} c_{13}= c_{13} a_{13}, \ a_{23} c_{23}= c_{23} a_{23}\rangle,$$ 
with stable letter $c_{12}$, associated subgroups $A = B = \langle a_{13} a_{23} \rangle$ and identity isomorphism $A \to B$.
\end{theorem} 

\begin{proof} Let $U_2$ be the infinite cyclic group generated by $a_{12}$. 
Define an epimorphism $\psi : ST_3 \to U_2$, by the rules
$$
\psi(a_{12}) = a_{12},~~ \psi(a_{13}) = \psi(a_{23}) = \psi(c_{12}) = \psi(c_{13}) = \psi(c_{23}) = 1.
$$
The kernel  $Ker(\psi)$ is the normal closure of the subgroup $H = \langle a_{13}, a_{23}, c_{12}, c_{13}, c_{23} \rangle$. From the defining relations of $ST_3$ follows that $H$ is normal in $ST_3$ and hence is equal to its normal closure. To find defining relations of $H$ we have to take relations
$$
a_{13} c_{13} = c_{13} a_{13},~~a_{23} c_{23} = c_{23} a_{23},~~c_{12}^{-1} (a_{13} a_{23}) c_{12} = a_{13} a_{23},
$$
and add all relations which we get after conjugations by $a_{12}^k$, $k \in \mathbb{Z}$. But it is not difficult to see that all these relations are equivalent to our three relations. Hence, $H$ has the presentation from theorem. 

The second part of the theorem follows from the definition of HNN-extension.
\end{proof}

\begin{theorem}
 $ST_3$ is isomorphic to $SP_3$.
\end{theorem}

\begin{proof}
We know a presentation for $SP_3$ from \cite[Theorem 3.9]{bk}.  We shall compare this presentation with that of $ST_3$ obtained above. 
Comparing the sets of relations for  $ST_3$ and  $SP_3$, we see that they are different by one relation. In $ST_3$ we have relation
$$
a_{12}^{-1} c_{23}a_{12}= a_{13}c_{23}a_{13}^{-1},
$$
but in $SP_3$ we have relation
$$
a_{12} b_{23}a_{12}^{-1} = a_{23}^{-1} a_{13}^{-1} b _{23 } a_{13}  a_{23}.
$$ 
Conjugating relation in $ST_3$ by $a_{12}^{-1}$ we get
$$
c_{23} = a_{13}^{a_{12}^{-1}} c _{23}^{a_{12}^{-1}} a_{13}^{-a_{12}^{-1}}.
$$
Using the defining relation of  $ST_3$ we have 
$$
c_{23} = a_{13}^{a_{23}} c _{23}^{a_{12}^{-1}} a_{13}^{-a_{23}}.
$$
Conjugating both sides of the last relation by $a_{13}^{a_{23}}$ we arrive to relation
$$
c_{23}^{a_{12}^{-1}} = a_{23}^{-1} a_{13}^{-1} (a_{23} c _{23} a_{23}^{-1}) a_{13} a_{23}.
$$
Since $a_{23}$ and $c_{23}$ are commute we have 
$$
c_{23}^{a_{12}^{-1}} = a_{23}^{-1} a_{13}^{-1}  c _{23}  a_{13} a_{23}.
$$
This relation is equivalent to relation in $SP_3$. Hence, the maps
$$
a_{ij} \mapsto a_{ij},~~~c_{ij}\mapsto b_{ij}
$$
define an isomorphism $ST_3 \to SP_3$.
\end{proof}

\bigskip 

Let us define some other decompositions of $ST_3$.

We know that $ST_3$ contains the pure braid group  $P_3 = \langle a_{12}, a_{13}, a_{23}  \rangle$ and $C_3 = \langle c_{12}, c_{13}, c_{23}  \rangle$. Define two maps
$$
\varphi_c : ST_3 \to P_3,~~\varphi_c(a_{ij}) = a_{ij},~~\varphi_c(c_{ij}) = e,
$$
$$
\varphi_a : ST_3 \to C_3,~~\varphi_a(a_{ij}) = e,~~\varphi_c(c_{ij}) = c_{ij}.
$$
From the defining relations of $ST_3$ follows that these maps define epimorphisms and we have two short exact sequences:
$$
1 \to Ker(\varphi_c) \to  ST_3 \to P_3 \to 1,
$$  
$$
1 \to Ker(\varphi_a) \to  ST_3 \to C_3 \to 1.
$$ 

It is easy to check that under $\varphi_a$ all relations of $ST_3$ go to the trivial relations. Hence, we have

\begin{proposition}
$C_3$ is the free group of rank 3.
\end{proposition}

We can find a generating set of $Ker(\varphi_c)$. Recall that $U_3 = \langle  a_{13}, a_{23}  \rangle$ is a free group of rank 2 which is normal in $P_3$ and $P_3$ is a semi-direct product of $U_3$ and infinite cyclic group $U_2 = \langle  a_{12}  \rangle$. Denote by $M_1$ the set of reduced words in the alphabet $\{ a_{13}^{\pm 1}, a_{23}^{\pm 1} \}$ which stated with some power of $a_{13}$.  Denote by $M_2$ the set of reduced words in the alphabet $\{ a_{13}^{\pm 1}, a_{23}^{\pm 1} \}$ which stated with some power of $a_{23}$.  Denote by $M_3$ the subset of $M_2$ consist of the word  which do not have the form $a_{23}^{-1} a_{13}^{-1} u$, where $u \in U_3$.

\begin{proposition} \label{p5}
The kernel $Ker(\varphi_c)$ is generated by elements
$$
c_{12}^u, c_{13}^v, c_{23}^w,~\mbox{where}~u \in M_3, v \in M_2, w \in M_1.
$$
\end{proposition}

\begin{proof}
By the definition $Ker(\varphi_c)$ is generated by elements $c_{ij}^w$, where $w \in P_3$. From the structure of $P_3$ follows, that $w = {a_{12}^k} w'$ for some integer $k$ and $w' \in U_3$. Using the conjugation rules by elements $a_{ij}$, we can assume that $Ker(\varphi_c)$ is generated by elements $c_{ij}^{w'}$, where $w' \in U_3$. Using the formulas (for $\varepsilon = \pm 1$):
$$
c_{12}^{a_{23}^{-1} a_{13}^{-1}} = c_{12},~~c_{12}^{a_{13}^{\varepsilon}} = c_{12}^{a_{23}^{-\varepsilon}},~~c_{13}^{a_{13}^{\varepsilon}} = c_{13},~~c_{23}^{a_{23}^{\varepsilon}} = c_{23},
$$
we get the need set of generators.
\end{proof}

\begin{question}
Is it true that $Ker(\varphi_c)$ is a free group with the set of free generators constructed in Proposition \ref{p5}?
\end{question}

\begin{ack}
The probelm addressed in this paper grew out during discussions with Valeriy G. Bardakov. The authors would like to thank Bardakov for numerous discussions and encouragements. 

Kozlovskaya acknowledges support by the Russian Science Foundation (project 19-41-02005) during the work on Section 3 and by the Ministry of Science and Higher Education of Russia (agreement No. 075-02-2020-1479/1) for the work in Section 4. 

Gongopadhyay acknowledges support from the DST projects \hbox{DST/INT/RUS/RSF/P-19} and  MTR/2017/000355.
\end{ack}


\begin{thebibliography}{D}

\bibitem[B92]{bae}
John C. Baez. 
\newblock  Link invariants of finite type and perturbation theory
\newblock {\em Lett. Math. Phys.}, 26(1):43--51, 1992.

\bibitem[B]{B} V. G. Bardakov, The virtual and universal braids,
\newblock {\em Fund. Math.}, 181 (2004), 1--18.

\bibitem[BK]{bk} V. G. Bardakov and T. Kozlovskaya, On 3-strand singular pure braid group, arXiv:2005.11751, to appear in J. Knot Theory Ramifications. 

\bibitem[Bi93]{bir}
Joan S. Birman. 
\newblock New points of view in knot theory. 
\newblock {\em Bull. Amer. Math. Soc. (N. S.)}, 28(2):253--287, 1993.
 

\bibitem[BB]{BB}
V. G. Bardakov, P. Bellingeri, {Combinatorial properties of virtual braids}, {\em Topology Appl.}, 156,
no. 6 (2009),  1071--1082.

\bibitem[CPM16]{cpm} Carmen Caprau,  Andrew de la Pena,  and Sarah McGahan,
\newblock Virtual singular braids and links
\newblock {\em Manuscripta Math.}, 151(1-2):147--175, 2016.

\bibitem[Co]{co} R. Corran, A normal form for a class of monoids including the singular braid monoid. {\em J. of Algebra}
223:256--282, 2000. 

\bibitem[DG98]{dg2} O. T. Dasbach, B. Gemein, The word problem for the singular braid monoid. arXiv:math/9809070.

 \bibitem[DG00]{dg1}O. T. Dasbach, B. Gemein, {\em A faithful representation of the singular braid monoid on three strands}. In Knots in Hellas '98 (Delphi), 48–58, Ser. Knots Everything, 24, World Sci. Publ., River Edge, NJ, 2000. 

\bibitem[Ja]{j}  A. Jarai, On the monoid of singular braids. {\em Topology Appl.} 96(2): 109--119, 1999. 
\newblock 
\bibitem[FKR98]{fkr}
Roger Fenn, Ebru Keyman, and Colin Rourke. 
\newblock The singular braid monoid embeds in a group.
\newblock {\em J. Knot Theory Ramifications}, Volume 7, no. 7, 881--892, 1998.

\bibitem[GP09]{pa}Eddy Godelle and Luis Paris, {\em On singular Artin monoids}. In Geometric methods in group theory. Contemp. Math. 372. 43--57, Amer. Math. Soc., Providence, RI. 2005. 

\bibitem[V14]{v} V. V. Vershinin, {\em About presentations of braid groups and their generalizations}. In Knots in Poland. III. Part 1, volume 100 of Banach Center Publ., pages 235--271. Polish Acad. Sci. Inst. Math.,
Warsaw, 2014.

\bibitem[Ra]{R}
L. Rabenda. Memoire de DEA (Master thesis), Universite de Bourgogne, 2003.


\end{thebibliography}
\end{document}